\documentclass[11pt,reqno]{amsart}
 \usepackage{amssymb}
 \usepackage{amsmath}
 \usepackage{mathrsfs}
    \usepackage{amssymb, amsmath, amsfonts}
    \usepackage[title]{appendix}
    \usepackage{mathabx}

    \usepackage{color}

    \usepackage[hidelinks]{hyperref}
    \allowdisplaybreaks
    \usepackage[numbers,sort&compress]{natbib}

    \makeatletter
    \def\tank#1{\protected@xdef\@thanks{\@thanks
     \protect\footnotetext[0]{#1}}}
    \def\bigfoot{

     \@footnotetext}
    \makeatother

    \renewcommand{\theequation}
    {\arabic{section}.\arabic{equation}}
    \newcommand{\ea}{\end{array}}

\DeclareMathOperator{\Var}{Var}

\allowdisplaybreaks
\numberwithin{equation}{section}

\allowdisplaybreaks

\newtheorem{theorem}{Theorem}[section]
\newtheorem{lemma}{Lemma}[section]
\newtheorem{proposition}[theorem]{Proposition}

\newtheorem{definition}[theorem]{Definition}

\newtheorem{remark}{Remark}[section]

\def\beq{\begin{equation}}
\def\nneq{\end{equation}}

\def\bthm{\begin{theorem}}
\def\nthm{\end{theorem}}

\def\blem{\begin{lemma}}
\def\nlem{\end{lemma}}
\def\bprf{\begin{proof}}
\def\nprf{\end{proof}}
\def\bprop{\begin{prop}}
\def\nprop{\end{prop}}
\def\brmk{\begin{rem}}
\def\nrmk{\end{rem}}

\def\bexa{\begin{exa}}
\def\nexa{\end{exa}}
\def\bcor{\begin{cor}}
\def\ncor{\end{cor}}

\def\RR{\mathbb{R}}

\def\EE{\mathbb{E}}
\def\bT{\mathbb{T}}
\def\bL{\mathbb{L}}

\def\cA{\mathcal{A}}

\def\cF{\mathcal{F}}

\def\cD{\mathcal{D}}

\def\LL{[-L,L]}

\newcommand\bp{\mathbb{P}}

\newcommand\diam { {\rm{diam\,}}}

\def\ee{{\mathbb E}}

\def\FF{\mathcal {F}}

\newcommand{\Blk}{\Big[}
\newcommand{\Brk}{\Big]}
\newcommand{\lc}{\left(}
\newcommand{\rc}{\right)}
\newcommand{\lk}{\left[}
\newcommand{\rk}{\right]}
\newcommand{\lt}{\left }
\newcommand{\rt}{\right}

\title[Growth rate for the  H\"older coefficients  of FSHE]{Growth rates for the  H\"older coefficients of   the linear stochastic fractional  heat equation  with  rough dependence in space}

    \author[C. Liu]{Chang Liu}
    \address[]{Chang Liu, School of Mathematics and Statistics,  Wuhan University,  Wuhan, 430072,
    China.}
    \email{changliu0504@163.com}

    \author[B. Qian]{Bin Qian}
    \address[]{Bin Qian, Department of Mathematics and statistics, Suzhou University of Technology, Changshu, Jiangsu 215500,   China.}
    \email{binqiancn@126.com}

    \author[R. Wang]{Ran Wang}
    \address[]{Ran Wang, School of Mathematics and Statistics,  Wuhan University,  Wuhan, 430072,
    China.}
    \email{rwang@whu.edu.cn}

    \date{}
    
\begin{document}
    \maketitle

     \noindent {\bf Abstract:} 
We study the linear stochastic fractional heat equation
$$
  \frac{\partial}{\partial t}u(t,x)=-(-\Delta)^{\frac{\alpha}2}u (t,x)+\dot{W}(t,x),\ \ t> 0,\  x\in\RR,
$$
where $-(-\Delta)^{\frac{\alpha}{2}}$ denotes the fractional Laplacian with  $\alpha\in (1, 2)$.  The driving noise $\dot W$ is a centered Gaussian field that is white in time and  has the covariance of a fractional Brownian motion with Hurst parameter $H\in\left(\frac {2-\alpha}2,\frac 12\right)$. We establish exact asymptotics for the solution as  both temporal and spatial variables tend to infinity,  and derive sharp growth rates for the H\"older coefficients. The proofs are based on Talagrand's majorizing measure theorem and Sudakov's minoration theorem.

 \vskip0.3cm


\section{Introduction and main results}
Consider the  following  linear stochastic fractional heat equation (SFHE, for short):
\begin{equation}\label{SFHE}
\begin{cases}
\frac{\partial}{\partial t}u(t,x) =-(-\Delta)^{\frac{\alpha}2}u(t,x)+ \dot{W}(t,x),\ \ t>0,\ x\in\RR, \\
u(0,x)\equiv 0.
\end{cases}
\end{equation}
Here, $-(-\Delta)^{\frac{\alpha}{2}}$ denotes the fractional Laplacian with   $\alpha\in (1,2)$,   $W(t,x)$ is a centered Gaussian field with the covariance given by
\begin{equation}\label{CovW}
\ee[W(t,x)W(s,y)]=\frac 12 \left(s\wedge t\right)\left( |x|^{2H}+|y|^{2H}-|x-y|^{2H} \right),
\end{equation}
for some $H\in\left(0, \frac12\right)$. Equivalently, the covariance of the noise $\dot{W}(t,x)=\frac{\partial^2 W}{\partial t\partial x}$ is
$$
\ee\left[\dot{W}(t,x)\dot{W}(s,y)\right]=\delta_0 (t-s)\Lambda\left(x-y\right),
$$
where $\Lambda$ is a distribution, whose Fourier transform is the measure $\mu(d\xi)=c_{H}|\xi|^{1-2H}d\xi$ with
\begin{align}
c_{H}= \,\frac{1}{2\pi}\Gamma(2H+1)\sin(\pi H).   \label{e.c1}
\end{align}

When $\alpha=2$ and $H=\frac{1}{2}$, Equation~\eqref{SFHE} reduces to the classical stochastic heat equation (SHE, for short) driven by space-time white noise. In this case, it is well-known that the stochastic heat equation admits a unique mild solution \cite{D1999,Walsh1986}.  
The study of stochastic partial differential equations  driven by fractional Brownian motion (fBm, for short) or other general Gaussian noises   has developed rapidly in recent years. For an overview, see, e.g., \cite{BT2008,HJQ19,HLT2019,LHW2022,LHW2023,LM22,T2014,B.RM2012,GSS24} and the references.
Among these works, Balan and Tudor \cite{BT2008} gave necessary and sufficient conditions on the Hurst index for the existence of solutions to the stochastic heat equation driven by the following Gaussian noises:
\begin{itemize}
\item fractional noise in time with Hurst index $H \in (1/2,1)$;
\item spatially correlated noise with covariance given by Riesz, Bessel, or Poisson kernel functions.
\end{itemize}
Herrell, Song, Wu, and Xiao \cite{HSWX2020} investigated the sample path regularity of solutions to the stochastic heat equation driven by fractional-colored Gaussian noise.

The global spatial behavior of  solutions to the SHE driven by space-time white noise and colored noise
has been studied by  Conus, Joseph, and Khoshnevisan  \cite[Theorem 1.2]{CJK2013} and by Conus, Joseph, Khoshnevisan, and Shiu   \cite[Theorem 2.3]{CJKS2013}. Recently, Hu and Wang \cite{HW2022} extended these results to spatial rough noise with covariance structure given by \eqref{CovW}. Specifically, for the SHE driven by Gaussian noise that is white in time and spatially fractional with Hurst index $H \in (\frac{1}{4}, \frac{1}{2})$, they established the  precise asymptotic behavior of solutions as both temporal and spatial variables tend to infinity. Furthermore, Hu and Wang \cite{HW2022} established sharp estimates for the domain-size dependence of the H\"older coefficients. Very recently, Liu, Hu, and Wang \cite{LHW2023} investigated analogous properties for solutions to the stochastic wave equation in arbitrary spatial dimensions, considering additive Gaussian noise with both temporal and spatial fractional structure.

In the present work, we extend the research program initiated in \cite{HW2022} to the fractional Laplacian framework for Equation \eqref{SFHE}. Specifically, our investigation focuses on two  aspects:
\begin{itemize}
\item the asymptotic behavior of solutions to Equation \eqref{SFHE};
\item the sharp growth rates of the H\"older regularity coefficients.
\end{itemize}

We introduce the following key function that characterizes the space-time scaling in our analysis:
\begin{equation}\label{Thm3.1-2}
\Psi(t, L) := 1 + \sqrt{\log_2\left( \frac{L}{t^{1/\alpha}} \vee 1 \right)}, \quad t > 0, \ L > 0.
\end{equation}
This function plays a critical role throughout our work, particularly in establishing the precise relationship between spatial and temporal variables.

First, we estimate sharp bounds for  $\mathbb E\left[\sup_{0\leq t\leq T, -L\leq x\leq  L}u(t,x)\right]$ and $\mathbb E\left[\sup_{-L\leq x\leq L}u(t,x)\right]$  as $T$ and  $L$ tend to infinity.

\begin{theorem}\label{Thm3.1}
\begin{enumerate}
\item[(a)]
There exist some positive constants $c_{1,1}$ and $c_{1,2}$  such that  for any $T$ and $L>0$,
  \begin{align}\label{Thm3.1-1}
c_{1,1} T^{\frac{2H+\alpha-2}{2\alpha}}\Psi(T,L) \le
&\, \EE \lk\sup_{\substack{0\le t\le T, \\ -L\le x\le L}}  u(t,x) \rk\le c_{1,2} T^{\frac{2H+\alpha-2}{2\alpha}}\Psi(T,L).
  \end{align}
\item[(b)]
There exist some  positive constants $c_{1,3}$ and $c_{1,4}$  such that for any $L\geq t^{\frac{1}{\alpha}}>0$,
\begin{equation}\label{Thm3.1-5}
c_{1,3} t^{\frac{2H+\alpha-2}{2\alpha}}\Psi(t,L)\leq
\EE\left[\sup_{-L\leq x\leq L} u(t,x)\right]\leq c_{1,4} t^{\frac{2H+\alpha-2}{2\alpha}}\Psi(t, L).
\end{equation}
\end{enumerate}
\end{theorem}

Next, we investigate the asymptotic behavior of the solution  under the following three  scaling regimes:
\begin{itemize}
\item[$(1)$] temporal and spatial domain size $T, L\rightarrow\infty$;
\item[$(2)$] spatial domain size $L\rightarrow\infty$;
\item[$(3)$] spatial variable  $|x|\rightarrow\infty$.
\end{itemize}

For any $T>0$ and $\delta>0$, let
\begin{equation}\label{eq Up}
    \Upsilon(T,  \delta):= [0,T]\times \left[-T^{\frac{1+\delta}{\alpha}},T^{\frac{1+\delta}{\alpha}}\right].
\end{equation}
We have the following asymptotics results. 
\begin{proposition}\label{cor1}
\begin{enumerate}
\item[(a)]    For any $\delta>0$, 
\begin{equation}\label{Thm3.1-3}
\begin{split}
c_{1,1}\sqrt{\frac{\delta}{\alpha}}&\, \le \liminf_{T\rightarrow\infty}\frac {  \sup_{(t,x)\in\Upsilon(T, \delta)} u(t,x) }    {T^{\frac{2H+\alpha-2}{2\alpha}} \sqrt{\log_2 T}  }\\
& \le  \limsup_{T\rightarrow\infty}\frac {  \sup_{(t,x)\in\Upsilon(T, \delta)} u(t,x) }    {T^{\frac{2H+\alpha-2}{2\alpha}} \sqrt{\log_2 T}  }\le c_{1,2}\sqrt{\frac{\delta}{\alpha}},\ \ \text{a.s.}
\end{split}
\end{equation}
\item[(b)] For any $t>0$,
\begin{equation}\label{cor-1}
\begin{split} 
c_{1,3}t^{\frac{2H+\alpha-2}{2\alpha}}
\leq&\, \liminf_{L\rightarrow\infty}\frac{\sup_{x\in[-L,L]}u(t,x)}{\sqrt{\log_2 L}} \\
\leq&\, \limsup_{L\rightarrow\infty}\frac{\sup_{x\in[-L,L]}u(t,x)}{\sqrt{\log_2 L}}
\leq c_{1,4}t^{\frac{2H+\alpha-2}{2\alpha}},\ \ \text{a.s.}
\end{split}
\end{equation}
\item[(c)]
For any $t>0$,
\begin{equation}\label{u/log}
\limsup_{|x|\rightarrow\infty}\frac{u(t,x)}{\sqrt{\log |x|}} =\sqrt{2c_{\alpha, H }}t^{\frac{2H+\alpha-2}{2\alpha}},\ \ \text{a.s.},
\end{equation}
where 
\begin{equation}\label{eq c a H}
c_{\alpha, H} :=\frac{c_{H}} {2H+\alpha-2}  2^{ \frac{2H+\alpha-2}{\alpha}}\Gamma\left(\frac{2-2H}{\alpha}\right).
\end{equation}
\end{enumerate}
Here, the constants $c_{1,i}, i=1, \cdots, 4$ are the constants given in   Theorem \ref{Thm3.1}.
\end{proposition}

\begin{remark}
 Proposition~\ref{cor1}(b) extends Theorem 1.2 of \cite{CJK2013}, which concerned the SHE  with space-time white noise,  to the more general case of the SFHE driven by spatially rough noise.  Proposition~\ref{cor1}(c) further unifies Theorem 2.3 of \cite{CJKS2013} and Equation (6.3)  of  \cite{KKX2017} within this same rough-noise framework. 
 For example, when $\alpha = 2$ and $H=1/2$,  \eqref{u/log} recovers the following  result of Khoshnevisan, Kim, and Xiao \cite[(6.3)]{KKX2017}:
\[
\limsup_{|x|\rightarrow\infty}\frac{u(t,x)}{\sqrt{\log |x|}} =\sqrt{2}\left(\frac{t}{\pi}\right)^{\frac14}, \quad \text{a.s.}
\] 
 
\end{remark}

Next, we investigate the precise domain-size dependence of the spatial H\"older regularity coefficient.
Denote
\beq\label{def:h:u}
\Delta_h u(t,x):=u(t,x+h)-u(t,x).
\nneq
\begin{theorem}\label{Thm3.2}
For any  $0\leq\theta\leq\frac{2H+\alpha-2}{2}$,
there exist some positive constants $c_{1,5}$ and $c_{1,6}$ such that
\beq\label{Thm3.2-1}
c_{1,5}|h|^{\frac{2H+\alpha-2}{2}}\Psi(t, L)
\le \, \EE \lk\sup_{-L\le x\le L} \Delta_h u(t,x) \rk\le c_{1,6}t^{\frac{2H+\alpha-2-2\theta}{2\alpha}}|h|^{\theta}\Psi(t, L),
\nneq
for all $L\geq t^{\frac{1}{\alpha}}$ and $0<|h|\le   \frac3{64} t^{\frac{1}{\alpha}}$.
\end{theorem}

\begin{proposition}\label{cor2}
\begin{enumerate}
\item[(a)]
For any $0\leq\theta\leq\frac{2H+\alpha-2}{2}$ and $0<|h|\le\frac3{64}t^{\frac{1}{\alpha}}$,
\begin{equation}\label{Thm3.2-2}
\begin{split}
c_{1,5}|h|^{\frac{2H+\alpha-2}{2}}
\leq&\, \liminf_{L\rightarrow\infty}\frac{ \sup_{-L\le x\le L} \Delta_h u(t,x)  }{\sqrt{\log_2 L}} \\
\leq&\, \limsup_{L\rightarrow\infty}\frac{ \sup_{-L\le x\le L} \Delta_h u(t,x)  }{\sqrt{\log_2 L}}\leq
 c_{1,6}t^{\frac{2H+\alpha-2-2\theta}{2\alpha}}|h|^\theta,\ \ \text{a.s.},
\end{split}
\end{equation}
where $c_{1,5}$ and $c_{1,6}$ denote  the constants given in   Theorem \ref{Thm3.2}.
\item[(b)] For any $t>0$ and $0<|h|\leq \frac3{64} t^{\frac{1}{\alpha}}$,
\begin{equation}\label{u:h/log}
\limsup_{|x|\rightarrow\infty}\frac{\Delta_h u(t,x)}{\sqrt{\log |x|}}
=\sqrt{2}\|\Delta_h u(t,0)\|_{L^2(\Omega)},\ \ \text{a.s.}
\end{equation}
\end{enumerate}
\end{proposition}

\begin{remark}
By Lemma \ref{metric} below, there exist some positive constants $c_{1,7}$ and $c_{1,8}$ such that for any $t>0$ and $0<|h|\leq \frac3{64} t^{\frac{1}{\alpha}}$,
\begin{equation*}
c_{1,7}|h|^{\frac{2H+\alpha-2}{2}}\leq\limsup_{|x|\rightarrow\infty}\frac{\Delta_h u(t,x)}{\sqrt{\log |x|}}\leq c_{1,8}|h|^{\frac{2H+\alpha-2}{2}},\ \ \text{a.s.}
\end{equation*}
\end{remark}
Finally, we establish sharp bounds for the domain-size dependence of the temporal H\"older regularity coefficient.
Denote
\beq\label{t:tau}
\mathcal D_{\tau} u(t,x):=u(t+\tau,x)-u(t,x).
\nneq
\begin{theorem}\label{Thm3.3}
For any   $0\leq\theta\leq\frac{2H+\alpha-2}{\alpha}$,
there exist some positive constants $c_{1,9}$ and $c_{1,10}$ such that
\beq\label{Thm3.3-1}
\aligned
c_{1,9}\tau^{\frac{2H+\alpha-2}{2\alpha}}\Psi(t, L)
\le&\, \EE \lk\sup_{-L\le x\le L} \mathcal D_{\tau} u(t,x) \rk\le c_{1,10}\tau^{\frac{\theta}{2}}t^{\frac{2H+\alpha-2-\alpha\theta}{2\alpha}}\Psi(t, L),
\endaligned
\nneq
for all $L\geq t^{\frac{1}{\alpha}}$ and $0<\tau\le \left(\frac3{32}\right)^\alpha t$.
\end{theorem}
\begin{proposition}\label{cor3}
\begin{enumerate}
\item[(a)]
For any $0\leq\theta\leq\frac{2H+\alpha-2}{\alpha}$ and $0<\tau\le \left(\frac3{32}\right)^\alpha t$,
\begin{equation}\label{Thm3.3-2}
\begin{split}
c_{1,9}\tau^{\frac{2H+\alpha-2}{2\alpha}}
\leq&\, \liminf_{L\rightarrow\infty}\frac{ \sup_{-L\le x\le L} \mathcal D_{\tau}u(t,x)  }{\sqrt{\log_2 L}} \\
\leq&\, \limsup_{L\rightarrow\infty}\frac{ \sup_{-L\le x\le L} \mathcal D_{\tau} u(t,x)  }{\sqrt{\log_2 L}}\leq
 c_{1,10}\tau^{\frac{\theta}{2}}t^{\frac{2H+\alpha-2-\alpha\theta}{2\alpha}},\ \ \text{a.s.},
\end{split}
\end{equation}
where $c_{1,9}$ and $c_{1,10}$   denote  the constants given  in Theorem \ref{Thm3.3}.
\item[(b)]
For any $t>0$ and $0<\tau\leq \left(\frac3{32}\right)^\alpha t$,
\begin{equation}\label{u:tau/log}
\limsup_{|x|\rightarrow\infty}\frac{\mathcal D_{\tau} u(t,x)}{\sqrt{\log |x|}}
=\sqrt{2}\|\mathcal D_{\tau} u(t,0)\|_{L^2(\Omega)},\ \ \text{a.s.}
\end{equation}
\end{enumerate}
\end{proposition}

\begin{remark}
By Lemma \ref{metric} below, there exist some positive constants $c_{1,11}$ and $c_{1,12}$ such that for any $t>0$ and $0<\tau\leq \left(\frac3{32}\right)^\alpha t$,
\begin{equation*}
c_{1,11}\tau^{\frac{2H+\alpha-2}{2\alpha}}\leq\limsup_{|x|\rightarrow\infty}\frac{\mathcal D_{\tau} u(t,x)}{\sqrt{\log |x|}}\leq c_{1,12}\tau^{\frac{2H+\alpha-2}{2\alpha}},\ \ \text{a.s.}
\end{equation*}
\end{remark}

The precise asymptotics and sharp growth rates for H\"older coefficients established in our main theorems are theoretically significant. Moreover, these results provide valuable tools for constructing power-decay functions, which have  direct applications to the well-posedness analysis of nonlinear  stochastic heat equation driven by rough noise, as demonstrated in \cite{HW2022}.  
Specifically, building upon the foundational work of Hu and Wang   \cite{HW2022}, Qian and Wang \cite{QW2025} established  the well-posedness of solutions to a nonlinear fractional stochastic heat equation in one spatial dimension:
\begin{equation}
\frac{\partial}{\partial t}u(t, x) = -(-\Delta)^{\alpha/2}u(t, x) + \sigma(t,x,u(t,x)) \dot{W}(t, x),
\end{equation}
where $\dot{W}$ denotes a space-time noise that is white in time and fractional in space with Hurst index $\frac{3-\alpha}{4} < H < \frac{1}{2}$ and  $1 < \alpha < 2$. Compared  with  the earlier  work of  Liu and Mao \cite{LM22}, the technical assumption $\sigma(0)=0$ is removed in \cite{QW2025}.

To construct  a suitable  decay weight function,  it is necessary  to derive lower bounds for the following semi-norm, which  quantifies the local variability of $u(t,x)$,
\begin{align}\label{eq norm}
\mathcal N_{\frac12-H}u(t,x):=\left(\int_{\RR}|u(t,x+h)-u(t,x)|^2 |h|^{2H-2}dh\right)^{\frac12}.
\end{align}

\begin{proposition} \label{E:sup:N} For any fixed $t>0$, there exists a positive constant $c_{t}$  depending on $t$ such that for any  $L\geq t^{\frac1\alpha}$,
\begin{equation}\label{E:sup:N:lower}
\mathbb E\left[\sup_{-L\leq x\leq L}\mathcal N^2_{\frac12-H}u(t,x)\right]\geq c_{t}\log_2(L).
\end{equation}
\end{proposition}

Note that when  $\alpha=2$,  the above estimates reduce to those obtained by Hu and Wang \cite{HW2022} for the stochastic heat equation with rough noise.

The remainder of the paper is organized as follows.   In Section 2, we  review  foundational properties of solutions and establish  sharp moment bounds. In Section 3, we  present the proofs of the main results, while technical lemmas are provided in the appendix.

\section{Moment estimates for the temporal and spatial increments}

\subsection{Stochastic integral and well-posedness of the solution}
Let us  recall some notations from \cite{HHLNT2017} and  \cite{HW2022}.   Let  $\cD(\RR)$  be the space of real-valued infinitely differentiable  functions with compact support on $\mathbb{R}$.      The Fourier transform of  a function $f\in\cD(\RR)$ is defined as
     $$
      \cF f(\xi):=\int_{\RR} e^{-i\xi x}f(x) dx.
    $$
   The Green kernel associated with the operator
$-(-\Delta)^{\frac{\alpha}{2}}$, where $\alpha\in (1,2]$, is denoted by  $G_{\alpha}$ and   defined via its Fourier transform 
    \begin{equation}\label{F-Green kernel}
    (\cF G_{\alpha}(t,\cdot))(\xi)=e^{-t|\xi|^{\alpha}}, \ \ \xi\in{\RR},\ \ t> 0.
    \end{equation}
See, e.g.,  \cite{ANTV22, CD15, CHN21}.

     Let $(\Omega,\cF,\mathbb P)$ be a complete probability space and  let $\mathcal{D}(\mathbb{R}_{+}\times \mathbb{R})$  denote  the space of real-valued, infinitely differentiable functions with compact support  in $\mathbb{R}_{+}\times \mathbb{R}$. The noise $\dot W$ is a zero-mean Gaussian family $\{W(\phi), \phi\in\cD(\RR_{+}\times \RR)\}$,  whose covariance structure is given by
     \begin{equation}\label{eq H product}
    \langle\phi,\psi\rangle_{\mathcal H}:= \EE\big[ W(\phi)W(\psi)\big]=c_{H}\int_{\RR_{+}\times \RR} \cF \phi(s,\xi) \overline{\cF \psi(s,\xi)}  |\xi|^{1-2H}d\xi ds,
     \end{equation}
     where $H\in\left(0, 1\right)$, $c_{H}$ is given in \eqref{e.c1},
     and $\cF \phi(s,\xi)$ is the Fourier transform of $\phi(s,x)$ with respect to the spatial variable $x$.
     
     Let $\mathcal H$  be the completion of $\mathcal D(\RR_+\times\RR)$ with the inner product
     $\langle\cdot,\cdot\rangle_{\mathcal H}$. The mapping $\varphi\mapsto W(\varphi)$ can be extended to all $\varphi\in\mathcal H$, which is called the Wiener integral.
\begin{definition}(\cite[Definition 4.1]{B2012})
The random field  $\{u(t,x); (t,x)\in \mathbb R_+\times \mathbb R\}$ defined by
\begin{equation}\label{u}
u(t, x)= \int_0^t \int_{\mathbb{R}}G_{\alpha}(t-s,x-y)  W(ds,dy)
\end{equation}
is called a  solution of Equation \eqref{SFHE}, provided that the stochastic integral on the right-hand side of \eqref{u} is well-defined. That is,  for every $t>0$ and $x\in\RR$, the function 
$$
g_{t,x}:=G_{\alpha}(t-\cdot,x-\cdot){\mathbf 1}_{[0,t]} 
$$
belongs to the space $\mathcal H$.
\end{definition}
  According to  Balan   \cite[Remark 5.4]{B2012},  the existence of a solution to Equation  \eqref{SFHE} can be stated as follows. 
\begin{proposition}\label{solution}  Equation \eqref{SFHE}  admits  a unique  solution $\{u(t,x); (t,x)\in \mathbb R_+\times \mathbb R\}$ of the form   \eqref{u} if and only if $$H\in\left(1-\frac{\alpha}{2},1\right).$$ In this case, for any $t>0$ and $x\in \mathbb R$, 
\begin{equation}\label{u-L2}
\|u(t,x)\|_{L^2(\Omega)}^2=c_{\alpha, H}t^{\frac{2H+\alpha-2}{\alpha}},
\end{equation}
where the constant  $c_{\alpha, H}$ is given by \eqref{eq c a H}.
\end{proposition}
\begin{proof}
Using \eqref{F-Green kernel},     for any $ \xi,x\in\RR$  and $t>0$, we have 
\begin{equation}\label{Fourier}
(\mathcal {F}G_{\alpha}(t,x-\cdot))(\xi)=e^{-t|\xi|^{\alpha}-i\xi x}.
\end{equation}
By \eqref{eq H product}, \eqref{Fourier}, and the change of variables
$\eta= 2(t-s) \xi^{\alpha}$, we obtain that   for all $x\in\RR$,
\begin{equation}\label{eq u moment}
\begin{split}
\EE\left[|u(t,x)|^2\right]
=&\,2c_{H}\int_0^t\int_0^{\infty}e^{-2(t-s)|\xi|^{\alpha}}|\xi|^{1-2H}d\xi ds \\
=&\, \frac{c_{H}}{\alpha}2^{\frac{2H+\alpha-2}{\alpha}}\int_0^\infty e^{-\eta }\eta^{\frac{2-2H}{\alpha}-1} d\eta \int_0^{t} (t-s)^{\frac{2H-2}{\alpha}} ds \\
=&\, \frac{c_{H}}{\alpha}2^{\frac{2H+\alpha-2}{\alpha}} \Gamma\left(\frac{2-2H}{\alpha}\right)\cdot \int_0^{t} (t-s)^{\frac{2H-2}{\alpha}} ds.
\end{split}
\end{equation}
Note that the last integral  converges if and only if $H> 1-\frac{\alpha}{2}$. In this case,   \eqref{u-L2} holds. The proof is complete.
\end{proof}

\subsection{H\"{o}lder continuity of  the solution}
In this part,  we give a sharp moment estimate for the increment
\begin{align}\label{eq d1}
d_1 ((t,x), (s,y)):=\left(\EE\left[| u(t,x)-  u(s,y)|^2\right]\right)^{\frac12}.
\end{align}
\begin{lemma}\label{metric} There exist some  positive constants $c_{2,1}$ and $c_{2, 2}$ such that for any  $(t,x)$, $(s,y)\in \RR_+\times\RR$,
\begin{equation}\label{equ-metric-d1}
c_{2,1}\tilde d_{1}((t,x),(s,y))\le  d_1 ((t,x), (s,y)) \le c_{2,2}\tilde d_{1}((t,x),(s,y)),
\end{equation}
  where
\begin{equation*}
\tilde d_{1}((t,x),(s,y))
:= |x-y|^{\frac{2H+\alpha-2}{2}}\wedge(t\wedge s)^{\frac{2H+\alpha-2}{2\alpha}}+|t-s|^{\frac{2H+\alpha-2}{2\alpha}}.
\end{equation*}
\end{lemma}
\begin{proof}
Without loss of generality, we assume $t>s>0$.
By \eqref{u}, we have
\begin{align*}
u(t,x)-u(s,y)=&\, \int_{0}^s\int_{\RR}\lk G_{\alpha}(t-r, x-z)-G_{\alpha}(s-r, y-z)\rk W(dr, dz)\\
&   +\int_{s}^t\int_{\RR} G_{\alpha}(t-r, x-z) W(dr, dz)\\
=&:I_1+I_2.
\end{align*}
The independence of the stochastic integrals over the time intervals $[0, s]$ and $[s, t]$  gives
\begin{equation}\label{d1 equality}
\EE\left[|u(t,x)-u(s,y)|^2\right]=\EE\left[I_1^2\right]+\EE\left[I_2^2\right].
\end{equation}

As in the proof of \eqref{eq u moment}, we obtain
\begin{equation}\label{E-I2}
\EE\left[I_2^2\right]= c_{\alpha, H}
(t-s)^{\frac {2H+\alpha-2}{\alpha}}.
\end{equation}

By \eqref{eq H product} and \eqref{Fourier},  we have
\begin{equation}\label{E-I1}
\begin{split}
\EE\left[I_1^2\right]=&\, c_{H}\int_{0}^s\int_{\RR}\lt|\FF G_{\alpha}(t-r, x-\cdot)(\xi)-\FF G_{\alpha}(s-r, y-\cdot)(\xi)\rt|^2   |\xi|^{1-2H}d\xi dr \\
 =& \, c_{H}\, \int_{0}^s\int_{\RR} e^{-2(s-r)|\xi|^{\alpha}}
\left[ 1+e^{-2(t-s)|\xi|^{\alpha}}-2e^{-(t-s)|\xi|^{\alpha}} \cos(\xi (x-y)) \right]  |\xi|^{1-2H}d\xi dr \\
= &\, c_{H}\int_{0}^{\infty} \lc 1-e^{-2s\xi^{\alpha}}\rc
     \left[ 1+e^{-2(t-s)\xi^{\alpha}}
       -2e^{-(t-s)\xi^{\alpha}}\cos(\xi(x-y))\right]   \xi^{1-2H-\alpha}d\xi.
       \end{split}
\end{equation}

Next, we derive matching upper and lower bounds for the above integral.

\noindent \textbf{Upper bound.}
By   \eqref{E-I1},  we have
\begin{equation}\label{x}
\begin{split}
\EE\left[\left|u(s,x)-u(s,y)\right|^2\right]
=&\,    2c_{H}\int_{0}^{\infty} \lc 1-e^{-2s\xi^{\alpha}}\rc
     \left[ 1-\cos(\xi(x-y))\right]  \xi^{1-2H-\alpha}d\xi \\
\leq&\,  2c_{H}  \int_{0}^{\infty}\left[ 1-\cos(\xi(x-y))\right]   \xi^{1-2H-\alpha}d\xi \\
=&\, 2c_{H} \frac{\Gamma(3-2H-\alpha)}{2H+\alpha-2}\cos\left(\left(H+\frac{\alpha}{2}-1\right)\pi\right) |x-y|^{2H+\alpha-2}.
\end{split}
\end{equation}
Here, the following identity  is used in the last step:
for any $\gamma\in(0,1)$ and $\xi\in\RR$,
\begin{equation}\label{cos}
\int_0^{\infty} \frac{1-\cos(\xi z)}{z^{1+\gamma}}dz= \frac{\Gamma(1-\gamma)}{\gamma}\cos\left(\frac{\pi\gamma}{2}\right) |\xi|^{\gamma}.
\end{equation}
See, e.g.,  \cite[Lemma D.1]{BJQ2015}.

By \eqref{d1 equality}, \eqref{E-I2}, \eqref{E-I1}, and
the change of variables $\eta=(t-s)\xi^\alpha$, we have
\begin{equation}\label{t}
\begin{split}
&\EE\left[\left|u(t,x)-u(s,x)\right|^2\right]\\
=&\, c_{H}\int_{0}^{\infty} \lc 1-e^{-2s\xi^{\alpha}}\rc
     \left(1-e^{-(t-s)\xi^{\alpha}}\right)^2   \xi^{1-2H-\alpha}d\xi
     +c_{\alpha, H}(t-s)^{\frac {2H+\alpha-2}{\alpha}} \\
\leq&\, c_{H}\int_{0}^{\infty}\left(1-e^{-(t-s)\xi^{\alpha}}\right)  \xi^{1-2H-\alpha}d\xi
     +c_{\alpha, H}(t-s)^{\frac {2H+\alpha-2}{\alpha}}\\
=&\, \frac{c_{H}}{\alpha}(t-s)^{\frac {2H+\alpha-2}{\alpha}}\int_{0}^{\infty}
     \left(1-e^{-\eta}\right) \eta^{\frac{2-2H-2\alpha}{\alpha}}d\eta
     +c_{\alpha, H}(t-s)^{\frac {2H+\alpha-2}{\alpha}} \\
=&\, \left(\frac{c_{H}}{2H+\alpha-2}\Gamma\left(\frac{2-2H}{\alpha}\right)+c_{\alpha,H}\right)
      (t-s)^{\frac {2H+\alpha-2}{\alpha}}.
\end{split}
\end{equation}

Using the  Minkowski inequality, \eqref{u-L2}, \eqref{x}, and \eqref{t},    there exists a constant $c_{2,3}>0$ such that
\begin{align*}
d_1 ((t,x), (s,y))
\le &\,    d_1 ((s,x), (s,y))+d_1 ((t,x), (s,x))\notag\\
\le &\, c_{2,3}\lc |x-y|^{\frac{2H+\alpha-2}{2}}\wedge s^{\frac{2H+\alpha-2}{2\alpha}}+(t-s)^{\frac{2H+\alpha-2}{2\alpha}}\rc,
\end{align*}
which is the upper bound in  \eqref{equ-metric-d1}.

\noindent
\textbf{Lower bound.}
By \eqref{E-I1} and the change of variables $\tilde{\xi}:=|x-y|\xi$, we have
\begin{equation}\label{I-1}
\begin{split}
\EE\left[I_1^2\right]
\geq&\,  c_{H}\int_{0}^{\infty} \lc 1-e^{-2s\xi^{\alpha}}\rc
       \left[ 1-e^{-(t-s)\xi^{\alpha}}\cos(\xi(x-y))\right]^2   \xi^{1-2H-\alpha}d\xi  \nonumber\\
       =&\,  c_{H}|x-y|^{2H+\alpha-2}  \cdot  \int_{0}^{\infty}\lk 1-\exp\lc-\frac{2s\tilde{\xi}^{\alpha}}{|x-y|^{\alpha}}\rc\rk  \\
   & \ \ \ \ \ \ \ \ \ \ \ \ \ \ \ \ \ \  \cdot \lk 1-\exp\lc-\frac{(t-s)\tilde{\xi}^{\alpha}}{|x-y|^{\alpha}}\rc\cos\left(\tilde{\xi}\right) \rk^2     \tilde{\xi}^{1-2H-\alpha}  d\tilde{\xi}\\
   \geq&\,  c_{H}|x-y|^{2H+\alpha-2}  \cdot   \sum_{n=1}^{\infty}  \int_{2n\pi+\frac{\pi}{2}}^{2n\pi+\frac{3\pi}{2}}  \lk 1-\exp\lc-\frac{2s\xi^{\alpha}}{|x-y|^{\alpha}}\rc\rk  \xi^{1-2H-\alpha}  d\xi,
\end{split}
\end{equation}
where the last inequality holds because   $\cos(\xi)\le 0$  on   the  interval $\left[2n\pi+\frac{\pi}{2},2n\pi+\frac{3\pi}{2}\right]$.

    Next, we derive a lower bound of  the above series.   For  convenience, we consider the following two cases:
 $$
     |x-y|\le s^{\frac{1}{\alpha}}\qquad{\rm and}\qquad |x-y|> s^{\frac{1}{\alpha}}.
 $$

\noindent {\bf Case 1.}    When $|x-y|\leq s^{\frac{1}{\alpha}}$, we have
	\begin{align}
  \EE\left[I_1^2\right] \geq   c_{H}|x-y|^{2H+\alpha-2}  \cdot   \sum_{n=1}^{\infty} \int_{2n\pi+\frac{\pi}{2}}^{2n\pi+\frac{3\pi}{2}}\lc 1-e^{-2\xi^{\alpha}} \rc  \xi^{1-2H-\alpha} d\xi. \label{I-2}
	\end{align}
Here, the series is convergent because  $\alpha+2H>2$.

\noindent {\bf Case 2.} 
The case  $|x-y|>s^{\frac{1}{\alpha}}$ is a little complicated. Denote
\begin{align}\label{eq n0}
	 n_0:=\inf\left\{n\in \mathbb N:2n\pi+\frac{\pi}{2}\geq\frac{\pi}{2} \frac{|x-y|}{s^{1/\alpha}} \right\}.
\end{align}
Then,
\begin{equation}\label{n_0}
\frac{\pi}{2} \frac{|x-y|}{s^{1/\alpha}}\leq
2n_0\pi+\frac{\pi}{2}\leq\frac{\pi}{2} \frac{|x-y|}{s^{1/\alpha}}+2\pi.
\end{equation}
 Since   for all  $\xi \geq 2n_0\pi + \frac{\pi}{2}$,
\[
1 - \exp\left(-\frac{2s\xi^{\alpha}}{|x-y|^{\alpha}}\right) \geq 1 - \exp(-2^{1-\alpha}\pi^{\alpha}), 
\]
 we have
\begin{equation}\label{eq:mathcalI2}
\begin{split}
\EE\left[I_1^2\right] &\geq c_{H}|x-y|^{2H+\alpha-2} \sum_{n=n_0}^{\infty} \int_{2n\pi+\frac{\pi}{2}}^{2n\pi+\frac{3\pi}{2}}
\left[1 - \exp\left(-\frac{2s\xi^{\alpha}}{|x-y|^{\alpha}}\right)\right] \xi^{1-2H-\alpha} d\xi \\
&\geq c_{H} \left(1 - \exp(-2^{1-\alpha}\pi^{\alpha})\right) |x-y|^{2H+\alpha-2}
\sum_{n=n_0}^\infty \int_{2n\pi+\frac{\pi}{2}}^{2n\pi+\frac{3\pi}{2}} \xi^{1-2H-\alpha} d\xi \\
&\geq c_{2,4} |x-y|^{2H+\alpha-2} \left(2n_0\pi+\frac{\pi}{2}\right)^{2-2H-\alpha},
\end{split}	
\end{equation}
where   the last inequality follows from the monotonicity of $\xi^{1-2H-\alpha}$ on $(0, \infty)$, and 
$$c_{2,4} = \frac{c_{H}(1 - \exp(-2^{1-\alpha}\pi^{\alpha}))}{2(2H+\alpha-2)}.$$

 When $|x-y|>s^{\frac{1}{\alpha}}$, by
 \eqref{n_0}, we have
\begin{equation}\label{I-3}
\begin{split}
  \EE\left[I_1^2\right] \ge &\,  c_{2,4} |x-y|^{2H+\alpha-2}{\lc2n_0\pi+\frac{\pi}{2}\rc}^{2-2H-\alpha}\\
   \geq&\, c_{2,4}
   |x-y|^{2H+\alpha-2}{\lc\frac{\pi}{2}\cdot\frac{|x-y|}{s^{1/\alpha}}+2\pi\rc}^{2-2H-\alpha}  \\
=&\, c_{2,4}\lc\frac{\pi}{2}\cdot\frac 1 {s^{1/\alpha}}+2\pi\cdot\frac 1 {|x-y|}\rc^{2-2H-\alpha} \\
\ge &\, c_{2,4} \lc\frac{5\pi}{2}\rc^{2-2H-\alpha}s^{\frac{2H+\alpha-2}{\alpha}}.
\end{split}
\end{equation}
Putting  \eqref{d1 equality}, \eqref{E-I2},  \eqref{I-1}, and \eqref{I-3} together, we have
\begin{equation*}
d_1 ((t,x), (s,y))
\ge  c_{2,5}\lc |x-y|^{\frac{2H+\alpha-2}{2}}\wedge s^{\frac{2H+\alpha-2}{2\alpha}}+(t-s)^{\frac{2H+\alpha-2}{2\alpha}}\rc,
\end{equation*}
which is the lower bound in  \eqref{equ-metric-d1}.
The proof is complete.
\end{proof}

\subsection{Moment bounds for the spatial increment}
Recall the operator $ \Delta_h$ defined by \eqref{def:h:u} for the  spatial increment.  We now  give the sharp bounds for  its natural metric:
    \begin{equation}\label{eq d2}
    d_{2,t,h}(x,y):=\left(\EE\left[| \Delta _h u(t,x)- \Delta _h u(t,y)|^2\right]\right)^{\frac12}.
  \end{equation}
\begin{lemma}\label{lem 2space} 
 \begin{itemize}
 \item[(a)](Upper bound)
 For any $0\leq\theta\leq\frac{2H+\alpha-2}{2}$,
 there exists a positive constant $c_{2,6}$ such that  for all $t, h\ge 0$ and $x, y\in \mathbb R$,
   \begin{align}\label{U-h-Eu-3}
   d_{2, t,h}(x,y)\leq c_{2,6}h^{\theta}\lc|x-y|^{\frac{2H+\alpha-2-2\theta}{2}}\wedge t^{\frac{2H+\alpha-2-2\theta}{2\alpha}}\rc.
 \end{align}
 \item[(b)](Lower bound)
  There exists a positive constant $c_{2,7}$ such that for all $|x-y|\geq t^{\frac1{\alpha}}$
  and
   $h\le \min\left\{\frac{3}{64}|x-y|, \frac{\pi}{4}\lc\frac{2}{\log 2}\rc^{\frac{1}{\alpha}}t^{\frac{1}{\alpha}}\right\}$,
  \begin{equation}\label{U-h-dEu}
       d_{2, t,h}(x,y)\geq  c_{2,7}h^{\frac{2H+\alpha-2}{2}}.
  \end{equation}
   \end{itemize}
 \end{lemma}
\begin{proof} The upper bound in \eqref{U-h-Eu-3} follows directly from \eqref{def:h:u}, Lemma \ref{metric}, and the Minkowski inequality. Let us prove the lower bound in \eqref{U-h-dEu}.

 Assume    $|x-y|\geq t^{\frac1{\alpha}}$ and $h\le \frac{3}{64}|x-y|$.    By \eqref{eq H product} and \eqref{Fourier}, we have
\begin{equation}
\begin{split}
d_{2,t,h}^2(x,y)
=&\, c_{H}\int_0^t\int_{\mathbb{R}} \left|
\mathcal{F} G_{\alpha}(t-s, x+h-\cdot)(\xi)-\mathcal{F} G_{\alpha}(t-s, x-\cdot)(\xi) \right. \\
& \quad\quad  \left. -\mathcal{F} G_{\alpha}(t-s, y+h-\cdot)(\xi) +\mathcal{F} G_{\alpha}(t-s, y-\cdot)(\xi) \right|^2  |\xi|^{1-2H}d\xi ds\\
=&\, c_{H}\int_0^t  \int_{\mathbb{R}} e^{-2(t-s)|\xi|^{\alpha}}  \left| 1 - e^{-ih\xi} \right|^2
   \left| 1 - e^{-i(x - y)\xi} \right|^2 |\xi|^{1-2H}d\xi  ds\\
=&\, 4c_{H}\int_{\mathbb{R}_+}\left( 1 - e^{-2t\xi^{\alpha}}\right)
   \left[ 1 - \cos(h\xi)\right]   \left[ 1 - \cos(|x - y|\xi)\right]  \xi^{1-2H - \alpha}d\xi.\label{eq-d2}
\end{split}
\end{equation}
Note that
for any $\xi\ge\frac{\pi|x-y|}{4h}$ and $h\leq\frac{\pi}{4}\lc\frac{2}{\log 2}\rc^{\frac{1}{\alpha}}t^{\frac{1}{\alpha}}$,
 \[
 1-\exp\lc-\frac{2t\xi^{\alpha}}{|x-y|^{\alpha}} \rc\geq \frac 12.
 \]
The elementary inequality
\begin{align}\label{eq elem}
1-\cos(x)\ge \frac{x^2}{4} \ \ \ \text{for } |x|\le \frac{\pi}{2},
\end{align}
implies that  for any $0<\xi\le \frac{|x-y|\pi}{2h}$,
\[
 1-\cos \left(\frac{h\xi}{|x-y|}\right)\ge
  \frac{h^2\xi^2}{4|x-y|^2}.
 \]
Therefore, by \eqref{eq-d2} and the change of variables $\tilde{\xi}:=|x-y|\xi$, we have
  \begin{equation}\label{d2:h}
    \begin{split}
       d_{2, t,h}^2(x,y)
        =&\, 4c_{H} |x-y|^{2H+\alpha-2} \int_{\RR_+}\lt[1-\exp\lc-\frac{2t\tilde{\xi}^{\alpha}}{|x-y|^{\alpha}}\rc\rt]\lt[1-\cos\lc \frac{h\tilde{\xi}}{|x-y|}\rc\rt] \\
        &\ \ \ \ \ \ \ \ \ \ \ \ \ \ \ \ \ \ \  \ \ \ \ \ \ \ \ \ \ \ \cdot\left[1-\cos\left(\tilde{\xi}\right)\right] \tilde{\xi}^{1-2H-\alpha}d\tilde{\xi} \\
        \geq&\,  2c_{H} |x-y|^{2H+\alpha-2}\int_{\frac{|x-y|\pi}{4h}}^{\infty} \lt[1-\cos\lc \frac{h\xi}{|x-y|}\rc\rt][1-\cos(\xi)] \xi^{1-2H-\alpha}d\xi \\
        \geq&\, \frac{c_{H}}{2} h^2|x-y|^{2H+\alpha-4}
        \int_{\frac{|x-y|\pi}{4h}}^{\frac{|x-y|\pi}{2h}} [1-\cos(\xi)] \xi^{3-2H-\alpha}d\xi.
    \end{split}
  \end{equation}

Define
\begin{align*}
   k_0:= &\, \inf \left\{k\in\mathbb Z_+:\frac{(6k+1)\pi}{3}\geq \frac{|x-y|\pi}{4h}\rt\},\\
      k_1:=&\, \sup\lt\{k\in\mathbb Z_+:\frac{(6k+5)\pi}{3}\leq \frac{|x-y|\pi}{2h}\rt\}.
  \end{align*}
  Equivalently,
$$
k_0=\Bigg\lfloor\frac{\frac{3|x-y|}{4h}-1}{6}\Bigg\rfloor+1,\ \ \
k_1=\Bigg\lfloor\frac{\frac{3|x-y|}{2h}-5}{6}\Bigg\rfloor.
$$
When $h\leq\frac3{64}|x-y|$,   we have  $k_0< k_1$. Define  
   $$I_k:= \, \left(\frac{(6k+1)\pi}{3},\frac{(6k+5)\pi}{3}\right].$$
 Then,
$$
\bigcup_{k=k_0}^{k_1}I_k\subseteq\left[\frac{|x-y|\pi}{4h},\frac{|x-y|\pi}{2h}\right].
$$
 Consequently,  we have
\begin{equation}\label{L-h-dEu0}
\begin{split}
\int_{\frac{|x-y|\pi}{4h}}^{\frac{|x-y|\pi}{2h}} \bigl(1-\cos\xi\bigr)  \xi^{3-2H-\alpha} \,d\xi
&\geq \sum_{k=k_0}^{k_1} \int_{I_k} \bigl(1-\cos\xi\bigr) \xi^{3-2H-\alpha} \,d\xi \\
&\geq \frac{1}{2}\sum_{k=k_0}^{k_1} \int_{I_k} \xi^{3-2H-\alpha} \,d\xi \\
&\geq \frac{1}{4} \int_{\frac{(6k_0+1)\pi}{3}}^{\frac{(6k_1+5)\pi}{3}} \xi^{3-2H-\alpha} \,d\xi.
\end{split}
\end{equation}
Here,   the  fact that $1 - \cos\xi \geq \frac{1}{2}$  for  $\xi\in I_k$  is used in the second inequality,
  and    the monotonicity of $\xi^{3-2H-\alpha}$ is used in the last step.

 From the definitions of $k_0$ and $k_1$,   when $h \leq \frac{3}{64}|x-y|$,
\begin{align*}
\frac{(6k_0+1)}{3} &\leq \frac{|x-y|}{4h} + 2 \leq \frac{11|x-y|}{32h},\\
\frac{(6k_1+5)}{3} &> \frac{|x-y|}{2h} - 2 \geq \frac{13|x-y|}{32h}. 
\end{align*}
Consequently, we have
\begin{equation}\label{L-h-dEu0-2}
\begin{split}
 \int_{\frac{(6k_0+1)\pi}{3}}^{\frac{(6k_1+5)\pi}{3}} \xi^{3-2H-\alpha}  d\xi 
= &\,  \frac{1}{4-2H-\alpha} \left[ \left(\frac{(6k_1+5)\pi}{3}\right)^{4-2H-\alpha} - \left(\frac{(6k_0+1)\pi}{3}\right)^{4-2H-\alpha} \right] \\
 \geq &\,  \frac{\pi^{4-2H-\alpha}}{4-2H-\alpha} \left[ \left(\frac{13}{32}\right)^{4-2H-\alpha} - \left(\frac{11}{32}\right)^{4-2H-\alpha} \right] \left(\frac{|x-y|}{h}\right)^{4-2H-\alpha}.
\end{split}
\end{equation}
From \eqref{d2:h}, \eqref{L-h-dEu0}, and \eqref{L-h-dEu0-2}, we obtain \eqref{U-h-dEu}. The proof is complete.
\end{proof}

\subsection{Moment bounds for the temporal increment}
Recall the temporal difference operator $\mathcal D_{\tau}$ defined in \eqref{t:tau}.   In this part,  we establish sharp bounds for its canonical metric defined by
\begin{align}\label{eq:d3}
d_{3,t,\tau}(x,y)  := \left(\mathbb{E}\left[|\mathcal D_{\tau} u(t,x) - \mathcal D_{\tau} u(t,y)|^2\right]\right)^{1/2},
\end{align}
where $x,y \in \mathbb{R}$ represent spatial coordinates and $t, \tau > 0$ denote time parameters.

 \begin{lemma}\label{lem:diff3} 
\begin{itemize}
\item[(a)] (Upper bound)   There exists a positive constant $c_{2,8}$ such that for all   $0 \leq \theta \leq \frac{2H + \alpha - 2}{\alpha}$,  $x, y \in \mathbb{R}$ and   $0 < \tau \leq t$,
 \begin{align}\label{U-h-dEu 3-0}
   d_{3, t,\tau}(x,y)\leq c_{2,8} \tau^{\frac{\theta}{2}} \left(|x-y|^{\frac{2H+\alpha-2-\alpha\theta}{2}}\wedge t^{\frac{2H+\alpha-2-\alpha\theta}{2\alpha}}\right).
 \end{align}
 \item[(b)] (Lower bound)
 There exists a positive constant $c_{2,9}$ such that for all $|x-y|\geq t^{\frac1{\alpha}}$
  and $0 < \tau \leq \left(\frac{3}{32}\right)^{\alpha}|x-y|^{\alpha}$,
\begin{equation}\label{U-h-dEu 3}
       d_{3, t,\tau}(x,y)\geq  c_{2,9} \tau^{\frac{2H+\alpha-2}{2\alpha}}.
  \end{equation}
  \end{itemize}
  \end{lemma}
 \begin{proof}
  The upper bound in  \eqref{U-h-dEu 3-0}   follows immediately from Lemma \ref{metric} and the Minkowski inequality. Now, we     prove  the lower bound  \eqref{U-h-dEu 3}.

Using the Fourier representation \eqref{eq H product} and \eqref{Fourier}, along with the independence of stochastic integrals over $[0,t]$ and $[t,t+\tau]$, we have
      \begin{equation}\label{eq-d3}
\begin{split}
d_{3, t,\tau}^2(x,y)
\ge&\, \mathbb{E} \Bigg[ \left| \int_{t}^{t+\tau}\int_{\mathbb{R}} \left[ G_{\alpha}(t+\tau - s,x - z) - G_{\alpha}(t+\tau - s,y - z) \right] W(ds,dz)\right|^2 \Bigg] \\
=&\,  c_{H}\int_t^{t+\tau}\int_{\RR}  \lt| 1-e^{i(x-y)\xi} \rt|^2   e^{-2(t+\tau-s)|\xi|^{\alpha}}    |\xi|^{1-2H} d\xi ds  \\
=&\, 2c_{H}\int_{\RR_+}\lk 1-e^{-2\tau\xi^{\alpha}}\rk\lk 1-\cos(\xi|x-y|)\rk \xi^{1-2H-\alpha}d\xi\\
=&\, 2c_{H}|x-y|^{2H+\alpha-2}\int_{\RR_+}\lk 1-\exp\lc-\frac{2\tau\tilde{\xi}^{\alpha}}{|x-y|^{\alpha}}\rc\rk
\lk 1-\cos(\tilde{\xi})\rk \tilde{\xi}^{1-2H-\alpha}d\tilde\xi\\
\ge &\, 2c_{H}|x-y|^{2H+\alpha-2} \int_{\frac{|x-y|\pi}{2\tau^{1/\alpha}}}^{\frac{|x-y|\pi}{\tau^{1/\alpha}}}\lk 1-\exp\lc-\frac{2\tau\xi^{\alpha}}{|x-y|^{\alpha}}\rc\rk
\lk 1-\cos(\xi)\rk \xi^{1-2H-\alpha}d\xi\\
\ge &\,   2  c_{H} \left(1- \exp\left(-2^{1-\alpha}\pi^{\alpha}\right)\right) |x-y|^{2H+\alpha-2} \int_{\frac{|x-y|\pi}{2\tau^{1/\alpha}}}^{\frac{|x-y|\pi}{\tau^{1/\alpha}}} \lk 1-\cos(\xi)\rk \xi^{1-2H-\alpha}d\xi,
 \end{split}
\end{equation}
 where the change of variables   $\xi \mapsto \tilde \xi/|x-y|$  is used in the  third-last step, and the elementary inequality
  \[
    1 - \exp\left(-\frac{2\tau\xi^\alpha}{|x-y|^\alpha}\right) \geq
    1- \exp\left(-2^{1-\alpha}\pi^{\alpha}\right),\ \ \text{if}\ \ \xi\geq \frac{|x-y|\pi}{2\tau^{1/\alpha}},
    \]
is used   in the last step.

For any  $  \xi \in  \left[\frac{|x-y|\pi}{2\tau^{1/\alpha}},  \frac{|x-y|\pi}{\tau^{1/\alpha}} \right]$, by \eqref{eq elem},  we have
    \[
  1 \ge    1 - \cos\left(\frac{\tau^{1/\alpha}\xi}{2|x-y|}\right) \ge \frac{\tau^{2/\alpha}\xi^2}{16|x-y|^2}.
    \]
    Consequently, when $\tau \leq \left(\frac{3}{32}\right)^\alpha   |x-y|^{\alpha}$, it holds that
     \begin{equation}\label{h-tau}
\begin{split}
   \int_{\frac{|x-y|\pi}{2\tau^{1/\alpha}}}^{\frac{|x-y|\pi}{\tau^{1/\alpha}}} \left( 1-\cos(\xi)\right) \xi^{1-2H-\alpha}d\xi
\ge &\,  \frac{1}{16}\tau^{\frac{2}{\alpha}}|x-y|^{-2}
\int_{\frac{|x-y|\pi}{2\tau^{1/\alpha}}}^{\frac{|x-y|\pi}{\tau^{1/\alpha}}}
\left( 1-\cos(\xi)\right) \xi^{3-2H-\alpha}d\xi\\
\geq&\, c_{2,10} \tau^{\frac{2H+\alpha-2}{\alpha}}|x-y|^{2-2H-\alpha},
      \end{split}
\end{equation}
   where $c_{2,10}>0$, and the last inequality follows from \eqref{L-h-dEu0} and \eqref{L-h-dEu0-2}.

Combining \eqref{eq-d3} and \eqref{h-tau}, we obtain \eqref{U-h-dEu 3}. The proof is complete.
 \end{proof}

\section{Proof of main results}
\subsection{Proof of Theorem \ref{Thm3.1}}
 \begin{proof}[Proof of Theorem \ref{Thm3.1}]
The proof follows the strategy of \cite[Theorem 1.1]{HW2022}. It hinges on two key probabilistic tools: Talagrand's majorizing measure theorem (Theorem \ref{Talagrand}) for the upper bound, and Sudakov's minoration theorem (Theorem \ref{Sudakov}) for the lower bound.

For  brevity,  let
  \[
 \bT:=[0,T]\quad{\rm and}\quad  \bL:=\LL.
\]

 \noindent \textbf{Upper bound.} We begin by constructing a dyadic partition  of the parameter space $\bT\times \bL$. For each integer $n \geq 1$, we define
\begin{itemize}
 \item  Temporal partition:
\begin{equation}\label{partition:T}
[0,T] = \bigcup_{j=0}^{2^{2^{n-1}}-1} \Bigl[j\cdot 2^{-2^{n-1}}T,\ (j+1)\cdot 2^{-2^{n-1}}T\Bigr).
\end{equation}
 \item  Spatial partition:
\begin{equation}\label{partition:L}
 [-L,L]=\bigcup\limits_{k=-2^{2^{n-2}}}^{2^{2^{n-2}}-1}\lt[k\cdot2^{-2^{n-2}}L,(k+1)\cdot2^{-2^{n-2}}L\rt).
\end{equation}
\end{itemize}
For any $(t,x)$ such  that
\begin{align*}
j  2^{-2^{n-1}}T  \leq t < (j+1)  2^{-2^{n-1}}T, \,\, \,
k  2^{-2^{n-2}}L  \leq x < (k+1)  2^{-2^{n-2}}L,
\end{align*}
we define
\[
A_n(t,x) := \Bigl[j 2^{-2^{n-1}}T,\ (j+1) 2^{-2^{n-1}}T\Bigr) \times \Bigl[k 2^{-2^{n-2}}L,\ (k+1)  2^{-2^{n-2}}L\Bigr).
\]

 By Theorem \ref{Talagrand}, we have
\begin{equation}\label{diam An-1}
    \mathbb{E}\left[\sup_{(t,x)\in\bT \times\bL} u(t,x)\right]
    \leq c_{3,1}\sup_{(t,x)\in\bT \times\bL} \sum_{n=0}^\infty 2^{n/2} \mathrm{diam}(A_n(t,x)),
\end{equation}
where $\mathrm{diam}(A_n(t,x))$ denotes the diameter  with respect to the metric $d_1$ defined in \eqref{eq d1}.

By  Lemma \ref{metric}, we obtain the following uniform bound for all $t \in [0,T]$ and $x \in \mathbb{R}$:
\begin{equation}\label{diam An-2}
\begin{split}
    \mathrm{diam}(A_n(t,x))
    \leq c_{2,2}\Bigl[\left(2^{-\left(\frac{2H+\alpha-2}{2}\right)2^{n-2}} L^{\frac{2H+\alpha-2}{2}}\right)
     \wedge T^{\frac{2H+\alpha-2}{2\alpha}}
    + 2^{-\left(\frac{2H+\alpha-2}{2\alpha}\right)2^{n-1}} T^{\frac{2H+\alpha-2}{2\alpha}}\Bigr].
\end{split}
\end{equation}

We proceed by considering two distinct cases, depending on the relationship between the spatial and temporal scales:
$$
     L < T^{1/\alpha}\qquad{\rm and}\qquad L \geq T^{1/\alpha}.
$$
\noindent {\bf Case 1.}   When $L<T^{\frac{1}{\alpha}}$, by \eqref{diam An-1} and \eqref{diam An-2}, we have
 \begin{equation}\label{upper-Eu-1}
 \begin{split}
  \EE \lk\sup_{(t,x)\in\bT\times\bL}u(t,x)\rk
  \leq&\, c_{3,2}T^{\frac{2H+\alpha-2}{2\alpha}}.
  \end{split}
\end{equation}
\noindent {\bf Case 2.}   When $L\geq T^{\frac{1}{\alpha}}$,
let
\begin{equation}\label{eq N0}
\begin{split}
N_0:=&\, \inf\left\{n\geq2: 2^{-2^{n-2}} \leq \frac{T^{\frac{1}{\alpha}}}{L}\wedge \frac12 \right\}\\
=&\, \inf\left\{n\geq2: 2^{n-2}\geq \log_2\left(\frac{L}{T^{\frac{1}{\alpha}}} \vee 2\right)\right\}.
\end{split}
\end{equation}
By \eqref{diam An-2} and the definition of $N_0$, we have
\begin{equation}\label{n=N0+1}
\begin{split}
&\sum_{n= N_0+1}^{\infty}2^{\frac{n}{2}} \diam(A_n(t,x)) \\
\leq&\, c_{2,2}\left(
\sum_{n= N_0+1}^{\infty}2^{\frac{n}{2}}\cdot 2^{-\lc \frac{2H+\alpha-2}{2}\rc \cdot 2^{n-2}} L^{\frac{2H+\alpha-2}{2}}
+\sum_{n= N_0+1}^{\infty}2^{\frac{n}{2}}\cdot
2^{-\lc\frac{2H+\alpha-2}{2\alpha}\rc  2^{n-1}} T^{\frac{2H+\alpha-2}{2\alpha}}\right) \\
\leq&\, c_{2,2}T^{\frac{2H+\alpha-2}{2\alpha}} \left[
\sum_{n= N_0+1}^{\infty}2^{\frac{n}{2}}
\left(\frac{2^{2^{N_0-2}}}{2^{2^{n-2}}}\right)^{\frac{2H+\alpha-2}{2}}
     + \sum_{n= N_0+1}^{\infty} 2^{\frac{n}{2}-\lc\frac{2H+\alpha-2}{2\alpha}\rc 2^{n-1}}\right] \\
\leq&\, c_{3,3}T^{\frac{2H+\alpha-2}{2\alpha}}.
  \end{split}
  \end{equation}
By the definitions  of $\Psi(T, L)$  in  \eqref{Thm3.1-2} and $N_0$ in \eqref{eq N0}, we have  
\begin{equation}\label{N0-upper}
2^{\frac{N_0}{2}}\leq2^{\frac32}\Psi(T, L).
\end{equation}
This, together with  \eqref{diam An-1} and \eqref{n=N0+1}, implies that
\begin{equation}\label{upper-Eu-2}
\begin{split}
 & \EE \lk\sup_{(t,x)\in\bT\times\bL}u(t,x)\rk\\
  \leq &\,   c_{3,1}\sup_{(t,x)\in\bT\times\bL}\sum_{n=0}^{N_0}2^{\frac{n}{2}}\diam(A_n(t,x))
    +c_{3,1} \sup_{(t,x)\in\bT\times\bL}\sum_{n= N_0+1}^{\infty}2^{\frac{n}{2}} \diam(A_n(t,x)) \\
    \leq& \, 2\sqrt{2}\lc\sqrt{2}+1\rc c_{3,1}c_{2, 2}T^{\frac{2H+\alpha-2}{2\alpha}}2^{\frac{N_0}{2}}
    +c_{3,1}c_{3,3}T^{\frac{2H+\alpha-2}{2\alpha}} \\
  \leq &\, \left(2^{\frac52}\sqrt{2}\lc\sqrt{2}+1\rc c_{3,1}c_{2,2}+c_{3,1}c_{3,3}\right) T^{\frac{2H+\alpha-2}{2\alpha}}\Psi(T, L).
  \end{split}
  \end{equation}

From \eqref{upper-Eu-1} and \eqref{upper-Eu-2}, we obtain the upper bound in \eqref{Thm3.1-1}.

\noindent \textbf{Lower bound.}
 When $L\geq T^{\frac{1}{\alpha}}$, we consider the sequence
$\left\{u(T,x_i),i=0,1,\cdots,\pm \Big\lfloor L/T^{\frac{1}{\alpha}}\Big\rfloor\right\}$, where
  \[
   x_j:=jT^{\frac{1}{\alpha}},\ \ \text{for}\ \ j=0,\pm1,\cdots,\pm\Big\lfloor L/T^{\frac{1}{\alpha}}\Big\rfloor.
  \]
By \eqref{equ-metric-d1}, we have that for any  $i\neq j$,
\[
   d_{1}((T,x_i),(T,x_j))\geq c_{2,1} T^{\frac{2H+\alpha-2}{2\alpha}}.
\]
Applying Theorem \ref{Sudakov}  yields
  \begin{equation}\label{lower-Eu-1}
  \begin{split}
     \EE \lk \sup_{(t,x)\in\bT\times\bL}u(t,x)\rk
     \geq&\, \EE \lk
     \sup_{-\big\lfloor L/T^{\frac{1}{\alpha}}\big\rfloor\le i\le\big\lfloor L/T^{\frac{1}{\alpha}}\big\rfloor}u(T,x_i)\rk\\ 
     \geq &\, \frac12 c_{A,1}c_{2,1} T^{\frac{2H+\alpha-2}{2\alpha}}\Psi(T,L).
  \end{split}
  \end{equation}

When $L< T^{\frac{1}{\alpha}}$, by \eqref{equ-metric-d1} and  Theorem \ref{Sudakov}, we have
\begin{equation}\label{lower-Eu-2}
\EE\lk\sup_{(t,x)\in\bT\times\bL}u(t,x)\rk\geq\EE\left[u\left(\frac{T}{2},0\right)\vee u(T,0)\right]\geq c_{A,1}c_{2,1} T^{\frac{2H+\alpha-2}{2\alpha}}.
\end{equation}
From \eqref{lower-Eu-1} and \eqref{lower-Eu-2}, 
we obtain  the lower bound in  \eqref{Thm3.1-1}.

The proof of \eqref{Thm3.1-5} is identical to that  of   \eqref{Thm3.1-1}  with the necessary modifications, and is omitted. This completes the proof.
\end{proof}

\subsection{Proof of Proposition \ref{cor1}}
\begin{proof}[Proof of Proposition \ref{cor1}]
 \noindent\textbf{(a)} Recall   $\Upsilon(T,\delta)$ given by \eqref{eq Up} for any $T, \delta > 0$. Define
\begin{align*}
f(T):=&\, \sup_{(t,x)\in\Upsilon(T,\delta)} u(t,x),\\ 
g(T):= &\, T^{\frac{2H+\alpha-2}{2\alpha}}\left( 1+\sqrt{\log_2\left(  T^{\frac{\delta}{\alpha}} \right)} \right).
\end{align*}
Given $\varepsilon>0$, we choose a subsequence $T(n)=n^{\beta}$ for some
$\beta>\frac{2\alpha c_{\alpha,H}\log 2 }{\delta c_{1,1}^2\varepsilon^2}$.
By \eqref{Thm3.1-1} and \eqref{u-L2}, we have
\begin{align*}
 \lambda_{H,\alpha}(n):=\EE \lk f\left(n^\beta\right)\rk
     \geq c_{1,1}g\left(n^\beta\right),
\end{align*}
    and
\begin{align*}
	\sigma^2_{H,\alpha}(n):= \, \sup_{(t,x)\in\Upsilon\left(n^\beta,\delta\right)}
     \EE\left[|u (t,x)|^2\right]
     = c_{\alpha,H}n^{\beta\lc \frac{2H+\alpha-2}{\alpha}\rc}.
\end{align*}
By Borell's inequality (Theorem \ref{Borell}), we obtain that for any $\varepsilon>0$,
\begin{equation*}
\begin{split}
   \bp \left(  \left|f\left(n^\beta\right)-
   \EE\left[ f\left(n^\beta\right) \right]  \right|
   \geq\varepsilon \EE\left[ f\left(n^\beta\right) \right]  \right)  
\leq  &\,  2  \exp\left( -\frac{\left(\varepsilon\lambda_{H,\alpha}(n)\right)^2}{2\sigma^2_{H,\alpha}(n)}\right) \\
\leq & \, 2 n^{-{ \frac{\beta\delta}{\alpha}\cdot\frac{c_{1,1}^2\varepsilon^2}{2c_{\alpha,H}\log 2}}}.
\end{split}
\end{equation*}
By Borel-Cantelli's lemma,  we have almost surely
\begin{equation}\label{lower-u-1}
\limsup_{n \rightarrow\infty}
\frac {  \left|f\left(n^\beta\right)-\EE\left[ f\left(n^\beta\right) \right]  \right| }
 {\EE\left[f\left(n^\beta\right) \right] }\leq \varepsilon.
	\end{equation}
	The monotonicity of the functions  $f$ and $g$ allows us to derive \eqref{Thm3.1-3} via standard interpolation techniques, using the bounds provided in  \eqref{Thm3.1-1} and \eqref{lower-u-1}.

\noindent\textbf{(b)}
 For any $\varepsilon>0$, we consider a subsequence $L^{\beta}(n)=n^\beta$ for some $\beta>\frac{2c_{\alpha,H}\log 2}{c_{1,3}^2\varepsilon^2}$ and denote $$\bL^\beta(n):=\left[-n^\beta,n^\beta\right].$$
 By
 \eqref {Thm3.1-5}, we have
\begin{align*}
\EE \lk\sup_{x\in\bL^\beta(n)} u(t,x)\rk
     \geq c_{1,3}t^{\frac{2H+\alpha-2}{2\alpha}}
     \lc 1+\sqrt{\log_2{\lc n^\beta/t^{\frac{1}{\alpha}}\rc}}\rc.
\end{align*}
The result \eqref{cor-1} follows by adapting the proof of \eqref{Thm3.1-3}, with minor modifications. Details are omitted for brevity.

 \noindent\textbf{(c)}  We prove  part (c) by verifying the conditions in Theorem \ref{Asymptotic}. 
 
 Let $\rho_1(x,y)$ denote the correlation function between $u(t,x)$ and $u(t,y)$, that is
\[
\rho_1(x,y): = \frac{\mathbb{E}[u(t,x) u(t,y)]}{\sqrt{\mathbb{E}[u(t,x)^2] \cdot \mathbb{E}[u(t,y)^2]}}.
\]
By the spatial stationarity of $\{u(t, x)\}_{x \in \mathbb{R}}$, $\rho_1(x,y)$ has the following representation in terms of $L^2$-increments:
\begin{equation}\label{rho1:increment}
\rho_1(x,y) = 1 - \frac{\|u(t,x) - u(t,y)\|_{L^2(\Omega)}^2}{2\mathbb{E}[u(t,x)^2]}.
\end{equation}
Combining Lemma \ref{metric} with \eqref{rho1:increment} yields,    for any $0 < |x - y| < t^{1/\alpha}$,
\begin{equation}\label{rho1:bounds}
1 - \frac{c_{2,2}^2}{2\|u(t,x)\|_{L^2(\Omega)}^2} |x - y|^{2H + \alpha - 2}
\leq \rho_1(x,y) \leq
1 - \frac{c_{2,1}^2}{2\|u(t,x)\|_{L^2(\Omega)}^2} |x - y|^{2H + \alpha - 2}.
\end{equation}
This verifies Condition (1) in Theorem \ref{Asymptotic}.

By \eqref{eq H product} and \eqref{Fourier}, we obtain the following representation for the correlation function:
\begin{equation}\label{rho1:1}
\begin{split}
\rho_1(x,y)
&= \frac{c_{H}}{2\|u(t,x)\|_{L^2(\Omega)}^2}
\int_{\mathbb{R}}\left(1-e^{-2t|\xi|^\alpha}\right) e^{-i(x-y)\xi} |\xi|^{1-2H-\alpha} d\xi \\
&= \frac{c_{H}}{\|u(t,x)\|_{L^2(\Omega)}^2}
\int_0^\infty \left(1-e^{-2t\xi^\alpha}\right) \xi^{1-2H-\alpha} \cos\left(\xi|x-y|\right) d\xi.
\end{split}
\end{equation}

Let
\begin{equation*}
h(\xi) := \left(1 - e^{-2t\xi^\alpha}\right)\xi^{1-2H-\alpha}\ \ \text{ for } \xi>0.
\end{equation*}
We  claim that
\begin{equation}\label{claim}
\int_0^\infty h(\xi) \cos\left(\xi|x-y|\right) d\xi = O\left(|x-y|^{\frac{H-1}{2}}\right), \quad \text{as} \quad |x-y|\to\infty.
\end{equation}
This asymptotic estimate verifies Condition (2) in Theorem~\ref{Asymptotic}. The proof of~\eqref{claim} will be presented  later; we first address the estimate in~\eqref{u/log}.

For a fixed parameter $\lambda > 0$, consider the function
\[
\phi_{\lambda}(x) := \sqrt{\lambda\log x}\ \ \ \text{for } x>0.
\] 
Let $I(\phi_{\lambda})$ denote the integral defined in \eqref{eq inte test}. We can prove the following dichotomy:
\begin{itemize}
\item If $\lambda > 2$, then $I(\phi_{\lambda}) < \infty$;
\item If $\lambda \leq 2$, then $I(\phi_{\lambda}) = \infty$.
\end{itemize}
An application of Theorem \ref{Asymptotic} consequently yields \eqref{u/log}.

It remains to verify the claim \eqref{claim}.  To do this,
 we divide the integral into two parts:
\begin{equation}\label{rho1:2}
\begin{split}
& \int_0^\infty h(\xi)  \cos\left(\xi|x-y|\right)d\xi \\
=&\, \int_0^{|x-y|^{-\frac{1}{4}}}h(\xi)\cos\left(\xi|x-y|\right)d\xi+
\int_{|x-y|^{-\frac{1}{4}}}^\infty h(\xi)\cos\left(\xi|x-y|\right)  d\xi \\
=&: J_1+J_2.
\end{split}
\end{equation}
Using the elementary inequality
\begin{equation}\label{e}
1-e^{-x}\leq x\ \ \text{for any}\   x\geq0,
\end{equation}
we have
\begin{equation}\label{rho1:3}
J_1\leq2t\int_0^{|x-y|^{-\frac{1}{4}}}\xi^{1-2H}d\xi=\frac{t}{1-H} |x-y|^{\frac{H-1}{2}}.
\end{equation}
Using the integration by parts formula and \eqref{e}, we have
\begin{equation}\label{rho1:4}
\begin{split}
J_2=&\, |x-y|^{-1}\int_{|x-y|^{-\frac{1}{4}}}^\infty h(\xi)d(\sin(\xi|x-y|)) \\
=&\, |x-y|^{-1}\left[
h(\xi)\sin(\xi|x-y|)\Big|_{|x-y|^{-\frac{1}{4}}}^\infty
   -\int_{|x-y|^{-\frac{1}{4}}}^\infty\sin(\xi|x-y|)d(h(\xi))\right] \\
\leq&\, |x-y|^{-1}\left[h\left(|x-y|^{-\frac{1}{4}}\right)+\int_{|x-y|^{-\frac{1}{4}}}^\infty |h'(\xi)|d\xi\right] \\
\leq&\, 2t|x-y|^{\frac{2H-5}{4}}+|x-y|^{-1}\int_{|x-y|^{-\frac{1}{4}}}^\infty |h'(\xi)|d\xi.
\end{split}
\end{equation}
Using the change of variables $\eta:=\sqrt{2t}\xi^{\frac{\alpha}{2}}$, we have
\begin{equation}\label{rho1:5}
\begin{split}
&\int_{|x-y|^{-\frac{1}{4}}}^\infty |h'(\xi)|d\xi\\
\leq&\,  2t\alpha\int_{|x-y|^{\frac{1}{4}}}^\infty e^{-2t\xi^\alpha}\xi^{-2H}d\xi+
     (2H+\alpha-1)\int_{|x-y|^{-\frac{1}{4}}}^\infty\left(1-e^{-2t\xi^\alpha}\right)\xi^{-2H-\alpha}d\xi \\
\leq&\,  2^{\frac{2H+2\alpha-1}{\alpha}} t^{\frac{2H+\alpha-1}{\alpha}}\int_0^\infty e^{-\eta^2} \eta^{\frac{2-4H-\alpha}{\alpha}}d\eta+(2H+\alpha-1)\int_{|x-y|^{-\frac{1}{4}}}^\infty\xi^{-2H-\alpha}d\xi \\
=&\, (2t)^{\frac{2H+\alpha-1}{\alpha}}\Gamma\left(\frac{1-2H}{\alpha}\right)+|x-y|^{\frac{2H+\alpha-1}{4}}.
\end{split}
\end{equation}
By \eqref{rho1:2}, \eqref{rho1:3}, \eqref{rho1:4}  and \eqref{rho1:5}, we have
\begin{align*}
&\int_0^\infty h(\xi) \cos\left(\xi|x-y|\right)d\xi\\
\leq&\, |x-y|^{\frac{H-1}{2}}\left(\frac{t}{1-H}+2t|x-y|^{-\frac34}
+(2t)^{\frac{2H+\alpha-1}{\alpha}}\Gamma\left(\frac{1-2H}{\alpha}\right)|x-y|^{-\frac{1+H}{2}}
+|x-y|^{-\frac{3-\alpha}{4}}\right),
\end{align*}
which implies \eqref{claim}.
The proof   is complete.
\end{proof}

\subsection{Proof of  Theorem  \ref{Thm3.2}}
\begin{proof}[Proof of  Theorem  \ref{Thm3.2}] 
The proof of Theorem \ref{Thm3.2} follows the approach of  Theorem 1.2 in \cite{HW2022} and   Theorem \ref{Thm3.1}. It relies on       Talagrand's majorizing measure theorem and Sudakov's minoration theorem.  The main steps are outlined below.

  For a fixed $L \geq t^{1/\alpha}$,  consider the partition of $[-L, L]$  introduced  in \eqref{partition:L}. Using    \eqref{U-h-Eu-3},   the diameter of $A_n(x)$ under the metric $d_{2,t,h}(x,y)$ defined in \eqref{eq d2}  satisfies 
\[
\diam(A_n(t,x))\le c_{2,6}h^{\theta}\lc (2^{-2^{n-2}}L)^{\frac{2H+\alpha-2-2\theta}{2}}\wedge t^{\frac{2H+\alpha-2-2\theta}{2\alpha}}\rc.
\]

 Let
 \begin{equation}\label{def:N1}
 N_1:=\inf\left\{n\geq2: 2^{-2^{n-2}}\leq \frac{ t^{\frac{1}{\alpha}} } { L}\wedge\frac12\right\}.
 \end{equation}

 Applying Theorem~\ref{Talagrand} in the same way as in the proofs of Theorem~\ref{Thm3.1} and   \eqref{N0-upper}, we obtain that for all  $\theta\in\Big[0,\frac{2H+\alpha-2}{2}\Big]$,
\begin{equation}\label{U-h-Eu}
\begin{split}
  \EE\lk\sup_{x\in\bL}\Delta_h u(t,x)\rk 
\leq&\, c_{3,1}\sup_{x\in\bL}\sum_{n=0}^{N_1}2^{\frac{n}{2}}\diam(A_n(t,x))
    +c_{3,1}\sup_{x\in\bL}\sum_{n= N_1+1}^{\infty}2^{\frac{n}{2}} \diam(A_n(t,x)) \\
\leq&\,  c_{3,1}c_{2,6}h^{\theta}t^{\frac{2H+\alpha-2-2\theta}{2\alpha}}
\left[ \sum_{n=0}^{N_1}2^{\frac{n}{2}}+\sum_{n=N_1+1}^{\infty}2^{\frac{n}{2}}
\left( \frac{2^{2^{N_1-2}}}{2^{2^{n-2}}}\right)^{\frac{2H+\alpha-2-2\theta}{2}}  \right] \\
\leq&\, c_{3,4}h^{\theta}t^{\frac{2H+\alpha-2-2\theta}{2\alpha}}\cdot2^{\frac{N_1}{2}} \\
\leq&\, 2^{\frac32}c_{3,4}h^{\theta}t^{\frac{2H+\alpha-2-2\theta}{2\alpha}}\Psi(t,L).
\end{split}
\end{equation}

Consider the interval $[-L, L]$ with $L \geq t^{1/\alpha}$.  We select the grid points
\[
x_j = j t^{1/\alpha}  \quad \text{for} \quad j = 0, \pm 1, \ldots, \pm\left\lfloor L/t^{1/\alpha} \right\rfloor.
\]
Applying Theorem~\ref{Sudakov} as in~\eqref{lower-Eu-1}, we obtain the lower bound for all  $0 < h \leq \frac{3}{64}t^{1/\alpha}$,
\begin{equation}\label{L-h-Eu}
\mathbb{E}\left[\sup_{-L \leq x \leq L} \Delta_h u(t,x)\right] \geq \frac{1}{2}c_{A,1}c_{2,7}h^{\frac{2H+\alpha-2}{2}}\Psi(t,L).
\end{equation}
Combining the upper bound~\eqref{U-h-Eu} with the lower bound~\eqref{L-h-Eu} yields the desired result~\eqref{Thm3.2-1}. The  proof is complete.
\end{proof}

\subsection{Proof of  Proposition \ref{cor2}}
\begin{proof}[Proof of  Proposition \ref{cor2}] 
  The result follows the argument outlined in the proof of Proposition \ref{cor1}; the key steps are summarized below. 
  
\noindent{\bf Step 1.}
Fix an arbitrary $\varepsilon>0$ and  suppose that   $\beta>\frac{2c_{2,2}^2\log 2 }{c_{1,5}^2\varepsilon^2}$. Denote
  $\bL^\beta(n) :=\left[-n^\beta,n^\beta\right]$.
By \eqref{Thm3.2-1} and \eqref{equ-metric-d1},
we have that for any $0<h\leq \frac3{64}t^{\frac{1}{\alpha}}$,
\begin{align*}
\EE\left[\sup_{x\in\bL^\beta(n)}\Delta_hu(t,x)\right]
\geq c_{1,5}h^{\frac{2H+\alpha-2}{2}}\lc 1+\sqrt{\log_2\lc\frac{n^\beta}{t^{1/\alpha}}\rc}\rc,
\end{align*}
and
\begin{align*}
     \sup_{x\in\bL^\beta(n)}\EE\left[|\Delta_hu (t,x)|^2\right]
     \leq c_{2,2}^2 h^{2H+\alpha-2}.
\end{align*}
Consequently,  combining Borell's inequality \eqref{lem3} with the Borel-Cantelli lemma and applying the estimate \eqref{Thm3.2-1}, we establish \eqref{Thm3.2-2}.

\noindent {\bf Step 2.}  Let $\rho_2(x,y)$ be  the correlation function   between $\Delta_h u(t,x)$ and $\Delta_h u(t,y)$, that is,
 $$\rho_2(x,y):=\frac{\mathbb E[\Delta_h u(t,x) \Delta_h u(t,y)]}{ \sqrt{\mathbb E[\Delta_h u(t,x)^2] \cdot \mathbb E[\Delta_h u(t,y)^2] }}.$$

Proceeding similarly to the proof of \eqref{rho1:bounds} and applying Lemma \ref{lem 2space}, we can find some positive constants $c_{3,5}, c_{3,6}$ and $c_{3,7}$ such that  for all $0<|x-y|<c_{3,5}h$,
 \begin{equation*}
1-\frac{c_{3,6}}{\|\Delta_hu(t,x)\|_{L^2(\Omega)}^2}|x-y|^{2H+\alpha-2}
\leq \rho_2(x,y)\leq
1-\frac{c_{3,7}}{\|\Delta_hu(t,x)\|_{L^2(\Omega)}^2}|x-y|^{2H+\alpha-2}.
\end{equation*}
 This verifies Condition $(1)$ in Theorem \ref{Asymptotic}.

By \eqref{eq H product}, \eqref{Fourier}, and \eqref{claim}, we have
 \begin{equation*}
 \begin{split}
\rho_2(x,y)
=&\, \frac{c_{H}}{\|\Delta_hu(t,x)\|_{L^2(\Omega)}^2}
\int_{\RR}\left( 1-e^{-2t|\xi|^\alpha}\right)
\left(1-\cos(h\xi)\right)|\xi|^{1-2H-\alpha}\cos(\xi|x-y|)d\xi \\
\leq&\, \frac{4c_{H}}{\|\Delta_hu(t,x)\|_{L^2(\Omega)}^2}
\int_{0}^\infty\left( 1-e^{-2t\xi^\alpha}\right)\xi^{1-2H-\alpha}\cos(\xi|x-y|)d\xi \\
=&\, O\left(|x-y|^{\frac{H-1}{2}}\right),\ \ \text{as}\ \ |x-y|\rightarrow\infty.
\end{split}
\end{equation*}
This verifies Condition (2) in Theorem \ref{Asymptotic}. The same argument as in  Proposition \ref{cor1}(c)   gives \eqref{u:h/log}. The proof is complete.
\end{proof}

\subsection{Proof of  Theorem \ref{Thm3.3}}
\begin{proof}[Proof of  Theorem \ref{Thm3.3}]
The proof of Theorem~\ref{Thm3.3} follows the strategy of  Theorem \ref{Thm3.1} and  \cite[Theorem 1.3]{HW2022}, combining Talagrand's majorizing measure theorem with Sudakov's minoration theorem. The main steps are outlined below.

For a fixed $L>0$, consider the partition of $[-L, L]$ introduced in \eqref{partition:L}. Under the metric $d_{3,t,\tau}(x,y)$ defined in \eqref{eq:d3}, the diameter of $A_n(x)$ is bounded  via Lemma \ref{lem:diff3}  as follows: 
\[
\diam(A_n(t,x))\le c_{2,8}\tau^{\frac{\theta}{2}}\lc (2^{-2^{n-2}}L)^{\frac{2H+\alpha-2-\alpha\theta}{2}}\wedge t^{\frac{2H+\alpha-2-\alpha\theta}{2\alpha}}\rc.
\]

For  $0<\tau\leq t$,  we apply  Theorem~\ref{Talagrand} in the same way as in Theorem~\ref{Thm3.2} together with the bound from \eqref{N0-upper} to obtain
\begin{equation}\label{U-t-Eu}
\begin{split}
\EE\lk\sup_{x\in\bL}\mathcal D_{\tau} u(t,x)\rk
\leq&\, c_{3,1}\sup_{x\in\bL}\sum_{n=0}^{N_1}2^{\frac{n}{2}}\diam(A_n(t,x))
    +c_{3,1} \sup_{x\in\bL}\sum_{n= N_1+1}^{\infty}2^{\frac{n}{2}} \diam(A_n(t,x)) \\
\leq&\, c_{3,1}c_{2,8}\tau^{\frac{\theta}{2}}t^{\frac{2H+\alpha-2-\alpha\theta}{2\alpha}}
\lk \sum_{n=0}^{N_1}2^{\frac{n}{2}}+\sum_{n=N_1}^{\infty}2^{\frac{n}{2}}
\lc \frac{2^{2^{N_1-2}}}{2^{2^{n-2}}}\rc^{\frac{2H+\alpha-2-\alpha\theta}{2}}  \rk \\
\leq&\, c_{3,8}\tau^{\frac{\theta}{2}}t^{\frac{2H+\alpha-2-\alpha\theta}{2\alpha}}\cdot2^{\frac{N_1}{2}} \\
\leq& \, 2^{\frac32}c_{3,8}\tau^{\frac{\theta}{2}}t^{\frac{2H+\alpha-2-\alpha\theta}{2\alpha}}\Psi(t,L),
\end{split}
\end{equation}
where $N_1$ is given by \eqref{def:N1}.

Following the argument in the proof of \eqref{lower-Eu-1} and applying Theorem \ref{Sudakov}, we obtain that for all $0 < \tau \leq \left(\frac{3}{32}\right)^\alpha t$,
\beq\label{L-t-Eu}
\EE\left[\sup_{-L\leq x\leq L}\mathcal D_{\tau} u(t,x)\right] \geq \frac{1}{2} c_{A,1}c_{2,7}\tau^{\frac{2H+\alpha-2}{2\alpha}}\Psi(t,L).
\nneq
Combining  \eqref{U-t-Eu} and \eqref{L-t-Eu}, we have   \eqref{Thm3.3-1}.  The proof is complete.
\end{proof}

\subsection{Proof of  Proposition \ref{cor3}}
\begin{proof}[Proof of  Proposition \ref{cor3}] 
The proof follows the structure of Proposition \ref{cor1}; only the key steps are presented here.

\noindent{\bf Step 1.}
For any $\varepsilon>0$, assume $\beta>\frac{2c_{2,2}^2\log 2}{c_{1,9}^2\varepsilon^2}$. Denote $L^\beta(n)=n^\beta$.
By \eqref{Thm3.3-1} and \eqref{equ-metric-d1}, we have that for any
$0<\tau\leq \left(\frac3{32}\right)^\alpha t$,
 \begin{equation*}
\EE\left[\sup_{x\in\bL^\beta(n)}\mathcal D_{\tau} u(t,x)\right]
\geq c_{1,9}\tau^{\frac{2H+\alpha-2}{2\alpha}}\lc 1+\sqrt{\log_2\lc\frac{n^\beta}{t^{1/\alpha}}\rc}\rc,
 \end{equation*}
and
 \begin{equation*}
\sup_{x\in\bL^\beta(n)}\EE\left[|\mathcal D_{\tau} u (t,x)|^2\right]
     \leq c_{2,2}^2\tau^{\frac{2H+\alpha-2}{\alpha}}.
 \end{equation*}
Using Borell's inequality \eqref{lem3} and the Borel-Cantelli lemma with \eqref{Thm3.3-1}, we obtain \eqref{Thm3.3-2}.

\noindent{\bf Step 2.}
To prove \eqref{u:tau/log}, we   verify the conditions of Theorem~\ref{Asymptotic}.   The correlation function between $\mathcal{D}_{\tau} u(t,x)$ and $\mathcal{D}_{\tau} u(t,y)$ is given by
 $$\rho_3(x,y):=\frac{\mathbb E[\mathcal D_{\tau} u(t,x) \mathcal D_{\tau} u(t,y)]}{ \sqrt{\mathbb E[\mathcal D_{\tau} u(t,x)^2] \cdot \mathbb E[\mathcal D_{\tau} u(t,y)^2] }}.$$

 Applying the approach of \eqref{rho1:bounds} through Lemma~\ref{lem:diff3},   there exist some  positive constants $c_{3,9}$, $c_{3,10}$, and $c_{3,11}$ such that for all $x,y$ with $0 < |x-y| < c_{3,9}h$,
\begin{equation*}
1-\frac{c_{3,10}}{\|\mathcal D_{\tau}u(t,x)\|_{L^2(\Omega)}^2}|x-y|^{2H+\alpha-2}
\leq \rho_3(x,y)\leq
1-\frac{c_{3,11}}{\|\mathcal D_{\tau}u(t,x)\|_{L^2(\Omega)}^2}|x-y|^{2H+\alpha-2},
\end{equation*}
where $\|\mathcal D_{\tau}u(t,x)\|_{L^2(\Omega)}^2$ does not depend on $x$; see \eqref{t}.

By \eqref{eq H product} and \eqref{Fourier},  we have
\begin{equation*}
\begin{split}
\rho_3(x,y)
=&\, \frac{c_{H}}{2\|\mathcal D_{\tau}u(t,x)\|_{L^2(\Omega)}^2}
\int_{\RR}\left[ \left(1-e^{-2t|\xi|^\alpha}\right)\left(1-e^{-\tau|\xi|^\alpha}\right)^2
+\left( 1-e^{-2\tau|\xi|^\alpha}\right) \right] \\
&\ \ \ \ \ \ \ \ \ \ \ \ \ \ \ \ \ \ \ \ \ \ \ \ \ \ \
    \cdot|\xi|^{1-2H-\alpha}\cos(\xi|x-y|)d\xi \\
\leq&\, \frac{c_{H}}{\|\mathcal D_{\tau}u(t,x)\|_{L^2(\Omega)}^2}
\Bigg[\int_0^\infty\left(1-e^{-2t\xi^\alpha}\right)\xi^{1-2H-\alpha}\cos(\xi|x-y|)d\xi \\
&\ \ \ \ \ \ \ \ \ \ \ \ \ \ \ \ \ \ \ \ \ \ \ \ \ +\int_0^\infty\left(1-e^{-2\tau\xi^\alpha}\right)\xi^{1-2H-\alpha}\cos(\xi|x-y|)d\xi\Bigg] \\
=&\, O\left(|x-y|^{\frac{H-1}{2}}\right),\ \ \text{as}\ \ |x-y|\rightarrow\infty,
\end{split}
\end{equation*}
where  \eqref{claim} is used in the last step.

Having verified the conditions of Theorem \ref{Asymptotic}, the estimate \eqref{u:tau/log} follows    in parallel to the proof of Proposition \ref{cor1}(c). This completes the proof.
\end{proof}

\subsection{Proof of  Proposition \ref{E:sup:N}}
 \begin{proof}[Proof of  Proposition \ref{E:sup:N}]
To prove \eqref{E:sup:N:lower}, we adapt the approach of \cite[Proposition 3.11]{HW2022} and apply the Sudakov minoration theorem.

Consider the   probability density function defined as 
$$
\varrho(h):=c_{3,12}\left(
|h|^{-\frac{H}2-\frac{\alpha}4} {\mathbf 1}_{\left\{|h|\leq1\right\}}+|h|^{2H-2}{\mathbf 1}_{\left\{|h|\geq1\right\}}\right),
$$
where $c_{3,12}=2^{-1}\left(\frac4{4-2H-\alpha}+\frac1{1-2H}\right)^{-1}$.
Using  the Cauchy-Schwarz inequality, we have
\begin{equation}\label{cauchy}
\begin{split}
\left|\int_{\RR} \Delta_h u(t,x) \varrho(h)dh\right|
\leq &\,
\left( \int_{\RR} |\Delta_h u(t,x)|^2\cdot |h|^{2H-2}dh\right)^{\frac12}\cdot
\left( \int_{\RR} \varrho^2(h)|h|^{2-2H} dh\right)^{\frac12}\\
=&\, c_{3,13}^{\frac12} \mathcal N_{\frac12-H}u(t,x),
\end{split}
\end{equation}
where $c_{3,13}=2c_{3,12}^2\left(\frac2{6-6H-\alpha}+\frac1{1-2H}\right)$.
This, together with Jensen's inequality, implies that
\begin{equation}\label{N:lower}
\begin{split}
\mathbb E\left[\sup_{-L\leq x \leq L}\mathcal N^2_{\frac12-H}u(t,x)\right]
\geq &\,
c_{3,13}^{-1}\mathbb E\left[ \left( \sup_{-L\leq x \leq L}\int_{\RR} \Delta_h u(t,x) \varrho(h)dh\right)^2\right]\\
\geq &\,
c_{3,13}^{-1}\left( \mathbb E\left[ \sup_{-L\leq x \leq L}\int_{\RR}\Delta_h u(t,x)\varrho(h)dh\right]\right)^2.
\end{split}
\end{equation}
Denote
\begin{equation*}
u_{\varrho}(t,x):= \, \int_{\RR} \Delta_h u(t,x) \varrho(h)dh.
\end{equation*}
By stochastic Fubini's theorem (see, e.g., \cite[Corollary 2.9]{K14}), we have 
\begin{equation*}
u_{\varrho}(t,x)=  \, \int_0^t\int_{\RR}\left( \int_{\RR} [G_\alpha(t-s,x+h-z)-G_\alpha(t-s,x-z)]\varrho(h)dh\right)W(ds,dz).
\end{equation*} 
By  \eqref{eq H product} and  \eqref{Fourier}, we have
\begin{equation*}
\begin{split}
\mathbb E\left[u^2_{\varrho}(t,x) \right]
=&\, c_{H}\int_0^t \int_{\RR}\left| \mathcal F \left(\int_{\RR}[G_\alpha(t-s,x+h-\cdot)-G_\alpha(t-s,x-\cdot)]
 \varrho(h)dh \right)(\xi)\right|^2 |\xi|^{1-2H}d\xi ds\\
=&\, c_{H}\int_0^t \int_{\RR}\left| \int_{\RR}\left[
 e^{-(t-s)|\xi|^\alpha-i(x+h)\xi}-e^{-(t-s)|\xi|^\alpha-ix\xi} \right]\varrho(h)dh\right|^2|\xi|^{1-2H}d\xi ds\\
=&\, c_{H}\int_0^\infty\left(1-e^{-2t\xi^\alpha}\right)\left( \int_{\RR}(1-\cos(h\xi))\varrho(h)dh\right)^2
 \xi^{1-2H-\alpha}d\xi\\
\leq&\, 4c_{H}\int_0^\infty\left(1-e^{-2t\xi^\alpha}\right)\xi^{1-2H-\alpha}d\xi.
\end{split}
\end{equation*}
The last integral is   convergent since  $\alpha+2H>2$, which guarantees that   $u_{\varrho}(t,x)$ is a well-defined Gaussian random field.

Next, we establish the lower  bound for the canonical metric $d_{t,\varrho}$, which is  defined by
\begin{equation}
d_{t,\varrho}:=\left( \mathbb E\left[ \left|u_{\varrho}(t,x)-u_{\varrho}(t,y)\right|^2\right]\right)^{\frac12}  \, \text{ for  }  |x-y|\geq t^{\frac1\alpha}.
\end{equation}  
The argument follows the approach used  in \eqref{lower-Eu-1}. 

Using \eqref{eq H product} and \eqref{Fourier}, we obtain
\begin{equation}\label{d4}
\begin{split}
 d^2_{t,\varrho}
=&\, c_{H}\int_0^t \int_{\RR}\Big| \mathcal F \Big(\int_{\RR}[G_\alpha(t-s,x+h-\cdot)-G_\alpha(t-s,x-\cdot)\\
&\,\ \ \ \ \ \ \ \ \ \ \ \ \ \    -G_\alpha(t-s,y+h-\cdot)+G_\alpha(t-s,y-\cdot)]
 \varrho(h)dh \Big)(\xi)\Big|^2 |\xi|^{1-2H}d\xi ds\\
=&\, c_{H}\int_0^t \int_{\RR}\left| \int_{\RR}
e^{-(t-s)|\xi|^\alpha}e^{-ix\xi}\left(1-e^{-i(y-x)\xi}\right)\left(e^{-ih\xi}-1\right)\varrho(h)dh\right|^2
|\xi|^{1-2H}d\xi ds\\
=&\, 4c_{H}\int_0^\infty\left(1-e^{-2t\xi^\alpha}\right)(1-\cos(|x-y|\xi))
\left(\int_0^\infty(1-\cos(h\xi))\varrho(h)dh\right)^2\xi^{1-2H-\alpha}d\xi.
\end{split}
\end{equation}
When $\xi\geq1$, by \eqref{eq elem}, we have
\begin{equation}\label{d4-2}
\begin{split}
\int_0^\infty(1-\cos(h\xi))\varrho(h)dh
\geq&\, c_{3,12}\int_0^1 (1-\cos(h\xi))h^{-\frac{H}2-\frac\alpha4}dh\\
=&\, c_{3,12}\xi^{\frac{H}2+\frac\alpha4-1}\int_0^{\xi}(1-\cos(h))h^{-\frac{H}2-\frac\alpha4}dh\\
\geq&\, c_{3,12}\xi^{\frac{H}2+\frac\alpha4-1}\int_0^1(1-\cos(h))h^{-\frac{H}2-\frac\alpha4}dh\\
\geq&\, \frac{c_{3,12}}{4}\xi^{\frac{H}2+\frac\alpha4-1}\int_0^1h^{-\frac{H}2-\frac\alpha4+2}dh\\
=&\, \frac{c_{3,12}}{12-2H-\alpha}\xi^{\frac{H}2+\frac\alpha4-1}.
\end{split}
\end{equation}
This, together with \eqref{d4}, implies that
\begin{equation}\label{d4:lower}
\begin{split}
d^2_{t,\varrho}
 \geq&\,  \frac{4c_{H}c_{3,12}^2}{(12-2H-\alpha)^2}\left(1-e^{-2t}\right)
\int_1^\infty (1-\cos(|x-y|\xi))\xi^{-H-\frac{\alpha}2-1}d\xi\\
=&\, \frac{4c_{H}c_{3,12}^2}{(12-2H-\alpha)^2}\left(1-e^{-2t}\right)|x-y|^{H+\frac{\alpha}2}
    \int_{|x-y|}^\infty(1-\cos(\xi))
      \xi^{-H-\frac{\alpha}2-1}d\xi.\\
\end{split}
\end{equation}
Define
\begin{equation}\label{k3}
 k_3:=\left\lfloor \frac{ |x-y|-\frac{\pi}{2}}{2\pi}\right\rfloor+1.
\end{equation}
Then, there exists a positive constant $c_{3,13}$ such that
\begin{equation}\label{k3-1}
\begin{split}
\int_{|x-y|}^\infty(1-\cos(\xi))
      \xi^{-H-\frac{\alpha}2-1}d\xi
\geq&\, \sum_{k=k_3}^\infty \int_{2k\pi+\frac{\pi}2}^{2k\pi+\frac{3\pi}2}
     (1-\cos(\xi))\xi^{-H-\frac{\alpha}2-1}d\xi  \\
\geq&\, \pi\sum_{k=k_3}^\infty\left(2k\pi+\frac{3\pi}2\right)^{-H-\frac{\alpha}2-1}\\
\ge &\, c_{3,13}\left(2k_3\pi+\frac{3\pi}2\right)^{-H-\frac{\alpha}2}\\
\geq&\,  c_{3,13}\left(|x-y|+3\pi\right)^{-H-\frac{\alpha}2},
\end{split}
\end{equation}
where the second inequality holds because   $\cos(\xi)$ is negative on the intervals  $ \big[2k\pi+\frac{\pi}{2},2k\pi+\frac{3\pi}{2}\big]$ and $\xi^{-H-\frac{\alpha}2-1}$ is a decreasing function on $(0,\infty)$.

Combining \eqref{k3-1} with \eqref{d4:lower},  we have   that for any $|x-y|\geq t^{\frac1\alpha}$,
\begin{equation}\label{d4:low}
\begin{split}
d^2_{t,\varrho}  \geq &\,  \frac{4c_{H} c_{3,13}c_{3,12}^2}{(12-2H-\alpha)^2}\left(1-e^{-2t}\right) \left( \frac1 {1+\frac{3\pi}{|x-y|}}\right)^{H+\frac{\alpha}2}\\
\geq &\, \frac{4c_{H} c_{3,13}c_{3,12}^2}{(12-2H-\alpha)^2}\left(1-e^{-2t}\right)  \left( \frac{ t^{\frac1\alpha} } {t^{\frac1\alpha}+3\pi}\right)^{H+\frac{\alpha}2}.
\end{split}
\end{equation}

For the   interval $[-L, L]$, where  $L \geq t^{1/\alpha}$, we define a set of  grid points
\[
x_j = j t^{1/\alpha}  \quad \text{for} \quad j = 0, \pm 1, \ldots, \pm\left\lfloor L/t^{1/\alpha} \right\rfloor.
\]
Using \eqref{d4:low} and  Theorem~\ref{Sudakov} as in~\eqref{lower-Eu-1}, there exists a positive constant $c_{t}$ depending on $t$ such that
\begin{equation}\label{E:h}
\mathbb E\left[ \sup_{-L\leq x\leq L}\int_{\RR}\Delta_h u(t,x)\varrho(h)dh\right]\geq c_t\sqrt{\log_2 L}.
\end{equation}
Combining with    \eqref{N:lower} and \eqref{E:h}, we obtain \eqref{E:sup:N}.
The proof is complete.
 \end{proof}

\vskip 0.5cm
{\bf Acknowledgments} \quad The authors are grateful to the anonymous referees for their constructive comments and corrections, which have led to significant improvement of this paper.  
\vskip 0.5cm
{\bf Conflict of Interest}\quad The authors declare no conflict of interest.

\renewcommand{\thesection}{A}  
\renewcommand{\thelemma}{A.\arabic{lemma}}  
\renewcommand{\theequation}{A.\arabic{equation}}  
\setcounter{equation}{0}  
\setcounter{theorem}{0}  
\section*{Appendix}  

 Let $\mathbb T$ be an index set and $\{X_t\}_{t\in \mathbb T}$ a centered Gaussian process. Define the canonical metric $d$   by
 	 $$d(t,s)=\sqrt{\EE\left[|X_t-X_s|^2\right]}.$$
\begin{theorem} {\rm (Talagrand's majorizing measure theorem, \cite[Theorem 2.4.1]{Talagrand2014})} \label{Talagrand} 
  There exists a constant $c>1$ such that  
 	\begin{equation}\label{lem1}
 	c^{-1}\inf_\cA \sup_{t\in \mathbb T} \sum_{n\geq 0}2^{\frac{n}{2}} \diam(A_n(t))\le \, \EE\Blk \sup_{t\in \mathbb T} X_t \Brk\leq\, c\inf_\cA \sup_{t\in \mathbb T} \sum_{n\geq 0}2^{\frac{n}{2}} \diam(A_n(t)),
 	\end{equation}
 	where   the infimum is taken over all increasing sequence $\cA:=\{\cA_n, n=1, 2, \cdots\}  $ of partitions of $\mathbb T$ such that $\#\cA_n  \leq 2^{2^n}$ ($\#A$ denotes the number of elements in the  set $A$),   $A_n(t)$ denotes the unique element of $\cA_n$ that contains $t$, and $\diam(A_n(t))$ is the diameter  of $A_n(t)$ with respect to the  metric $d$.
 \end{theorem}

\begin{theorem}{\rm (Sudakov's minoration theorem, \cite[Lemma 2.4.2]{Talagrand2014})} \label{Sudakov}
Let $\{X_{t_i},i=1,\cdots,n\}$ be a centered Gaussian family with natural distance $d$ and assume that
\[
\forall p,q \leq n,~p\neq q \Rightarrow d(t_p,t_q)\geq \delta.
\]
Then, we have
\begin{equation}\label{lem2}
\EE\left[\sup_{1\leq i\leq n} X_{t_i} \right] \geq  c_{A,1} \delta   \sqrt{\log_2(n)},
\end{equation}
where $c_{A,1}$ is a universal constant.
\end{theorem}

\begin{theorem} { \rm (Borell's inequality, \cite[Theorem 2.1]{adler})}\label{Borell}
Let $\{X_t\}_{t\in \mathbb T}$  be a centered separable Gaussian process on some topological index set $\mathbb T$
with almost surely bounded sample paths. Then $\EE\left[\sup_{t\in  \mathbb T} X_t\right]<\infty$, and for all $\lambda>0$
	\begin{equation}\label{lem3}
	\bp\lt(\lt|\sup_{t\in \mathbb T} X_t-\EE\left[\sup_{t\in \mathbb T} X_t\right]\rt|>\lambda\rt)\leq 2\exp\lc-\frac {\lambda^2}{2\sigma^2}\rc,
	\end{equation}
	where $\sigma^2:=\sup_{t\in \mathbb T}\EE\left[X_t^2\right]$.
 \end{theorem}

\begin{theorem}{\rm (\cite[Theorems 1.1, 2.1 and 2.2]{QW1971})}\label{Asymptotic}
Let $\{X(t)\}_{t\in\mathbb{R}}$ be a real-valued separable stationary Gaussian process with $\EE[X(t)] \equiv 0$ and $\Var[X(t)] \equiv 1$. Assume that  its correlation function $\rho(s,t) = \EE[X(s)X(t)]$ is continuous and  satisfies the following conditions:
\begin{enumerate}
\item[(1)]
There are some  positive constants $\delta, c_{A,2}, c_{A,3}, T$, and $\alpha\in (0,2]$ such that
$$
1-c_{A,2}h^\alpha \leq \rho(t,t+h) \leq 1-c_{A,3} h^\alpha,\ \ \text{for all}\ \ 0\le h\le \delta, \  \ t\ge T;
$$
\item[(2)]
There exists a constant $\gamma>0$ such that
$$
\rho(t,t+h)=O(h^{-\gamma}),\ \ \text{as}\ \ h\rightarrow\infty.
$$
\end{enumerate}
Then, for every   positive  non-decreasing function $\phi(t): [a,\infty)\rightarrow \mathbb R_+$,
$$
 \bp\Big( \exists t_0(\omega)>a: X(t)<\phi(t)\ \ \text{for all}\ \ t\geq t_0(\omega)\Big)=1\ \ \text{or}\ \ 0,
$$
according as the integral
\begin{align}\label{eq inte test}
I(\phi):=\int_a^\infty {\phi(t)}^{2/\alpha-1}\exp\left( -\phi(t)^2/2\right)dt
\end{align}
converges or diverges. 
\end{theorem}

\end{document}